\documentclass[onefignum,onetabnum]{siamart190516}

\usepackage{lipsum}
\usepackage{amsfonts}
\usepackage{graphicx}
\usepackage{epstopdf}
\usepackage{algorithmic}
\usepackage[utf8]{inputenc}
\usepackage{enumitem}
\ifpdf
  \DeclareGraphicsExtensions{.eps,.pdf,.png,.jpg}
\else
  \DeclareGraphicsExtensions{.eps}
\fi

\newsiamremark{remark}{Remark}
\newsiamremark{hypothesis}{Hypothesis}
\crefname{hypothesis}{Hypothesis}{Hypotheses}
\newsiamthm{claim}{Claim}

\headers{SDC methods for second-order problems}{Akramov et al.}

\title{Spectral deferred correction methods for second-order problems\thanks{Submitted to the editors DATE.
}}

\author{Ikrom Akramov\thanks{Lehrstuhl Computational Mathematics, Technische Universität Hamburg, Hamburg, Germany}
\and Sebastian Götschel\footnotemark[2]
\and Michael Minion\thanks{Lawrence Berkeley National Lab, Berkeley, CA 94720, USA}
\and Daniel Ruprecht\footnotemark[2]
\and Robert Speck\thanks{Jülich Supercomputing Centre, Forschungszentrum Jülich GmbH, Jülich, Germany}
}

\usepackage{amsopn}

\ifpdf
\hypersetup{
  pdftitle={TITLE},
  pdfauthor={AUTHORS}
}
\fi

\newcommand{\RR}{\mathbb{R}}

\newcommand{\Tt}[1]{\mathbf{#1}}
\newcommand{\Td}[1]{#1}
\newcommand{\rev}[1]{{#1}}

\usepackage{verbatim}
\usepackage{marginnote}
\begin{document}

\maketitle

\begin{abstract}
Spectral deferred corrections (SDC) are a class of iterative methods for the numerical solution of ordinary differential equations. 
SDC can be interpreted as a Picard iteration to solve a fully implicit collocation problem, preconditioned with a low order method. 
It has been widely studied for first-order problems, using explicit, implicit or implicit-explicit Euler and other low order methods as preconditioner. 
For first-order problems, SDC achieves arbitrary order of accuracy and possesses good stability properties.
While numerical results for SDC applied to the second-order Lorentz equations exist, no theoretical results are available for SDC applied to second-order problems.

We present an analysis of the convergence and stability properties of SDC using velocity-Verlet as the base method for general second-order initial value problems.
Our analysis proves that the order of convergence depends on whether the force in the system depends on the velocity.
We also demonstrate that the SDC iteration is stable under certain conditions. 
Finally, we show that SDC can be computationally more efficient than a simple Picard iteration or a fourth-order Runge-Kutta-Nyström method.
\end{abstract}



\section{Introduction}\label{sec:intro}

Many problems in science and engineering can be modeled using Newton's second law, giving rise to initial value problems of the form
\begin{equation}\label{eq:ode}
	\ddot{x}=f(t, x(t), \dot{x}(t)), \ x(t_{0})=x_{0},\ \dot{x}(t_{0})=\dot{x}_{0},
\end{equation}
where $x:\RR \rightarrow \RR^{d}$, $\Td{f}: \RR \times \RR^{d} \times \RR^{d}\rightarrow \RR^{d}$, and $t_{0}\leq t \leq t_{\mathrm{end}}$. 
Only for very simple problems is it possible to find solutions analytically. 
In most cases, numerical time stepping algorithms must be used to generate approximate
solutions.
A straightforward approach is to rewrite~\eqref{eq:ode} as 
\begin{subequations}
\label{eq:ode_system}
\begin{align}
    \dot{x}(t) &= v(t) \label{eq:ode_system:x}  \\
    \dot{v}(t) &= f(t, x(t), v(t)) \label{eq:ode_system:v}
\end{align}
\end{subequations}
and apply standard methods for first-order ODEs like Runge-Kutta or multi-step methods.
This, however, means treating both the equation for the position $x(t)$ and $v(t)$ in the same way and can forfeit opportunities to improve method performance.

For the harmonic oscillator with $f(t,x(t), v(t)) = -x(t)$, for example, explicit Euler is unconditionally unstable while implicit Euler leads to heavy numerical damping.
By contrast, the symplectic Euler method, which integrates~\eqref{eq:ode_system:x} explicitly and~\eqref{eq:ode_system:v} implicitly (although no implicit solver is required if $f$ does not depend on $v$), is conditionally stable and energy conserving~\cite{HairerEtAl2003}.
A generalization of this approach are Runge-Kutta-Nystrom (RKN) methods that use different Butcher tables for position and velocity~\cite[Sec II.2]{HairerEtAl2002_geometric}.
Derivation of higher-order RKN methods leads to a quickly growing number of order conditions~\cite[Sec III.3.2]{HairerEtAl2002_geometric}.
A widely studied special case~\cite{Blanes2002,ImoniEtAl2006,Houwen1987,Houwen1989} is where $f(t,x(t),v(t)) = f(x(t))$, that is, the right hand side depends only on the position but not on the velocity.
This greatly simplifies order conditions, allowing for easy construction of high order methods with favourable properties~\cite[Sec III.2.3]{HairerEtAl2002_geometric}.
For the general case, however, constructing good high order methods remains a challenge and RKN often struggle to outperform standard Runge-Kutta methods designed for first-order problems~\cite{BaiJunkins2011}.

Spectral deferred corrections (SDC), introduced in 2000 by Dutt et al.~\cite{DuttEtAl2000}, provide an easy way to construct high order methods for first-order problems.
There is a significant amount of theory~\cite{ChristliebEtAl2010_MoC,HagstromZhou2006,HansenStrain2011}, a number of algorithmic improvements~\cite{HuangEtAl2006,Minion2003,Minion2004} and studies of its performance in complex applications~\cite{BouzarthMinion2010,Shen2023}.
For second-order problems, only a special variant based on the Boris integrator~\cite{Boris1970} has been proposed, which is specifically tailored to the Lorentz equations modelling trajectories of charged particles in electro-magnetic fields~\cite{WinkelEtAl2015}.
This Boris-SDC method has been improved~\cite{TretiakRuprecht2019}, used to compute fast ion trajectories in fusion reactors~\cite{TretiakEtAl2021} and studied for other plasma physics problems~\cite{SmedtEtAl2023}.
However, no attempts have been made to adopt SDC for second-order problems other than the Lorentz system and, unlike the first-order case, no theoretical foundation exists.

This paper fills this gap by providing a systematic study of the mathematical properties of second-order SDC, including a proof of consistency and an assessment of stability.
It studies convergence and demonstrates that SDC can compete with a RKN-4 method in terms of computational efficiency.
Section~\ref{sec:method} describes the SDC method for second-order IVPs using a velocity-Verlet integrator as base method.
Section~\ref{sec:stability} investigates stability, using the damped harmonic oscillator as a test problem.
The related issues of stability and convergence of the SDC iteration are discussed and stability domains of SDC are compared against a RKN-4 method and a collocation method using Picard iterations.
Section~\ref{sec:convergence} proves consistency and that each iteration increases the order by two in the case where $f$ does not depend on $v$ but only by one if it does.
The theoretical statements on convergence order are validated against numerical examples.
Finally, Section~\ref{sec:efficiency} compares the computational efficiency of SDC against Picard iteration and RKN-4.
All the numerical examples were produced with the project \texttt{Second\_orderSDC} in the pySDC software~\cite{Speck2019}, which is publicly available~\cite{speck_2024_10637162}.

\section{Spectral deferred corrections for second-order problems}\label{sec:method}
For the sake of notational simplicity, we focus on the autonomous case of equation \eqref{eq:ode_system} since any non-autonomous problem can be transformed into an equivalent autonomous problem~\cite[pages 6-7]{Chicone2006}.
Formulation of second-order SDC as well as notation are based on the description of the Boris-SDC algorithm by Winkel et al.~\cite{WinkelEtAl2015}.

\subsection{Collocation formulation}\label{section:collocation}
Consider~\eqref{eq:ode_system} in integral form
\begin{subequations}\label{eq:Picard_int_eq}
	\begin{align}
		&\Td{x}(t)= \Td{x}_{0}+ \int_{t_{n}}^{t}\Td{v}(s)d s,\\
		&\Td{v}(t)=\Td{v}_{0}+\int_{t_{n}}^{t}\Td{f}(\Td{x}(s), \Td{v}(s))d s,
	\end{align}
\end{subequations}
over a time step $[t_{n}, t_{n+1}]$ with starting values $\Td{x}_{0} \approx \Td{x}(t_{n})$ and $\Td{v}_{0} \approx \Td{v}(t_{n})$.
Then, define a set of quadrature nodes
\begin{equation*}
	t_{n}\leq \tau_{1}< \dots < \tau_{M} \leq t_{n+1},
\end{equation*}
with associated weights
\begin{equation*}
\Delta t q_{m,j}	=\Delta t\int_{0}^{e_{m}}l_{j}(s)d s=\int_{t_{n}}^{\tau_{m}}\bar{l}_{j}(s)ds,  \ m, j=1, \dots, M,
\end{equation*}
where $\Delta t=t_{n+1}-t_{n}$, $l_{j}(s)$ and $\bar{l}_{j}(s),\ j=1,\dots, M$ are Lagrange polynomials corresponding to the quadrature nodes on the intervals $[0,1]$ and $[t_{n}, t_{n+1}]$ respectively.
By $\Td{x}_{j},\ \Td{v}_{j},\ \Td{f_{j}}$ we denote numerical approximations to $\Td{x}(\tau_{j}),\ \Td{v}(\tau_{j})$ and $\Td{f}(\Td{x}(\tau_{j}), \Td{v}(\tau_{j}))$~\cite[p. 211-214]{HairerEtAl1993_nonstiff}. 
Approximating the integrals in~\eqref{eq:Picard_int_eq} using quadrature we obtain
\begin{subequations}\label{eq:coll_main}
\begin{align}\label{eq:coll_main_pos}
&\Td{x}_{m}=\Td{x}_{0}+\Delta t\sum_{j=1}^{M}q_{m,j}\Td{v}_{j},\\\label{eq:coll_main_vel}
&\Td{v}_{m}=\Td{v}_{0}+\Delta t\sum_{j=1}^{M}q_{m,j}\Td{f}_{j},  
\end{align}
\end{subequations}
for $m=1, \dots, M.$
Next, substitute the second equation in~\eqref{eq:coll_main} into the first so that
\begin{subequations}\label{eq:coll_SDC}
\begin{align}\label{eq:coll_SDC_pos}
&\Td{x}_{m}=\Td{x}_{0}+\Delta t \sum_{j=1}^{M}q_{m,j}\Td{v}_{0}+\Delta t^{2}\sum_{j=1}^{M}qq_{m,j}\Td{f}_{j},\\
&\Td{v}_{m}=\Td{v}_{0}+\Delta t\sum_{j=1}^{M}q_{m,j}\Td{f}_{j}\label{eq:coll_SDC_vel}
\end{align}
\end{subequations}
with $qq_{m,j}=\sum_{i=1}^{M}q_{m,i}q_{i,j}$ and $m=1, \ldots, M$.
The $x_m$, $v_m$ correspond to the stages of a fully implicit Runge-Kutta-Nyström method~\cite[pp. 283--300]{HairerEtAl1993_nonstiff}. 
We use Gauss-Legendre nodes throughout this paper, making the collocation method symplectic~\cite[Theorem 4.2]{HairerEtAl2002_geometric}.

\paragraph{Collocation in matrix form} For the purpose of analysis, we will write the $M$ coupled equations~\eqref{eq:coll_SDC} as a single system.
Let $\bar{Q}\in\RR^{M\times M}$ have entries $q_{m,j}$ and let
\begin{equation}
\Tt{V}=(\Td{v}_{0}, \Td{v}_{1}, \dots, \Td{v}_{M})^{T}, \ \Tt{X}=(\Td{x}_{0}, \Td{x}_{1}, \dots, \Td{x}_{M})^{T}\in \RR^{d(M+1)}
\end{equation}
be vectors that contain the approximations at all nodes\footnote{We use boldface variables to indicate values that have been aggregated over multiple quadrature nodes. However, note that non-boldface variables can be vectors, too. For example, $v_1 \in \mathbb{R}^d$ is the velocity at the first quadrature node whereas $\Tt{V} \in \RR^{d(M+1)}$ are the velocities at all quadrature nodes.}. With initial conditions
\begin{equation*}
\Tt{X}_{0}:=(\Td{x}_{0}, \Td{x}_{0}, \dots, \Td{x}_{0})^{T}, \ \Tt{V}_{0}:=(\Td{v}_{0}, \Td{v}_{0}, \dots, \Td{v}_{0})^{T}\in \RR^{d(M+1)}	
\end{equation*}
and $F(\Tt{X}, \Tt{V})=(\Td{f}_{0}, \Td{f}_{1}, \dots, \Td{f}_{M})^{T}\in \RR^{d(M+1)}$ denoting the vector that contains the forces at each node,
equation~\eqref{eq:coll_SDC} can be written compactly as
\begin{subequations}
\label{eq:dis_coll}
\begin{align}
    &\Tt{X}=\Tt{X}_{0}+\Delta t \Tt{Q}\Tt{V}_{0}+\Delta t^{2}\Tt{QQ}F(\Tt{X}, \Tt{V}),\label{eq:dis_coll_pos}\\ 
    &\Tt{V}=\Tt{V_{0}}+\Delta t \Tt{Q}F(\Tt{X}, \Tt{V}) \label{eq:dis_coll_vel}.
\end{align}
\end{subequations}
Here,
\begin{equation*}
Q:=\left(\begin{matrix}
	0 & \Tt{0} \\
	\Tt{0} & \bar{Q}
\end{matrix}\right)\in \RR^{(M+1)\times (M+1)}
\end{equation*}
with $\Tt{0}$ being the $M-$dimensional zero-vector, and
\begin{equation}
\Tt{Q}=Q\otimes \Tt{I}_{d} , \ \Tt{QQ}=(Q\otimes \Tt{I}_{d})\otimes (Q\otimes \Tt{I}_{d})=QQ \otimes \Tt{I}_{d}
\end{equation} 
with $\Tt{I}_{d}$ being the identity matrix of dimension $d$.
Finally, let
\begin{equation*}
    \Tt{U}=(\Tt{X}, \Tt{V}) = (\Td{x}_{0}, \dots, \Td{x}_{M}, \Td{v}_{0}, \dots, \Td{v}_{M})^{T}\in \RR^{2d(M+1)}.
\end{equation*}
Then, the equations~\eqref{eq:dis_coll} can be written as
\begin{equation}
    \label{eq:coll_matrix}
   \Tt{C}_{\textrm{\textrm{coll}}}\Tt{U}_{0} =  \Tt{U} - \Delta t\Tt{Q}_{\textrm{\textrm{coll}}}\Tt{F}(\Tt{U}) =: \Tt{M}_{\textrm{coll}}(\Tt{U}),
\end{equation}
with $\Tt{U}_{0}=(\Td{x}_{0}, \dots, \Td{x}_{0}, \Td{v}_{0},\dots, \Td{v}_{0})^{T},\ \Tt{F}(\Tt{U})=(\Td{f}_{0},\dots, \Td{f}_{M},\Td{f}_{0}, \dots, \Td{f}_{M})^{T}\in \RR^{2d(M+1)}$ and 
\begin{equation}\label{eq:Qcoll}
    \Tt{Q}_{\textrm{coll}}=\left(
    \begin{matrix}
    \Delta t\Tt{QQ} & \Tt{O}\\
    \Tt{O} & \Tt{Q}
    \end{matrix}
    \right), \ 
    \Tt{C}_{\textrm{coll}}=\left(
    \begin{matrix}
    \Tt{I}_{d(M+1)} & \Delta t\Tt{Q}\\
    \Tt{O} & \Tt{I}_{d(M+1)}
    \end{matrix}
    \right)\in \RR^{2d(M+1)\times 2d(M+1)}
\end{equation}
where $\Tt{O}$ denotes $d(M+1)\times d(M+1)-$dimensional matrix with zero entries.
Using~\eqref{eq:coll_matrix} we obtain the collocation problem in operator form
\begin{equation}\label{eq:coll_matrix_final}
    \Tt{M}_{\textrm{coll}}(\Tt{U})=\Tt{C}_{\textrm{coll}}\Tt{U}_{0}.
\end{equation}

Once the stages $x_m$, $v_m$ are known, the approximations at the end of the time step can be computed via
\begin{subequations}\label{eq:up_main}
	\begin{align}\label{eq:up_main_pos}
		&\Td{x}(t_{n+1})\approx x_{n+1}=\Td{x}_{0}+\Delta t\sum_{m=1}^{M}q_{m}\Td{v}_{m},\\\label{eq:up_main_vel}
		&\Td{v}(t_{n+1})\approx v_{n+1}= \Td{v}_{0}+\Delta t \sum_{m=1}^{M}q_{m}\Td{f}_{m}
	\end{align}
\end{subequations}
where
\begin{equation*}
	q_{j}=\int_{0}^{1}l_{j}(s)d s, \ j=1, \dots, M.	
\end{equation*}
Insert~\eqref{eq:coll_SDC_vel} into~\eqref{eq:up_main_pos} to obtain
\begin{subequations}\label{eq:up_SDC}
	\begin{align}
		&\Td{x}_{n+1}= \Td{x}_{0}+\Delta t \Td{v}_{0}+\Delta t^{2}\sum_{m=1}^{M}\sum_{i=1}^{M}q_{i}q_{i,m}\Td{f}_{m},\\
		&\Td{v}_{n+1}= \Td{v}_{0}+\Delta t \sum_{m=1}^{M}q_{m}\Td{f}_{m}
	\end{align}
\end{subequations}
where we use that $\sum_{m=1}^{M}q_{m}=1$ by consistency of the quadrature rule. 
Equations~\eqref{eq:up_SDC} can  again be written in vector form
\begin{subequations}\label{eq:up_vec}
	\begin{align}\label{eq:vec_x}
		&\Td{x}_{n+1}= \Td{x}_0+\Delta t \Tt{q} \Tt{V}_{0}+\Delta t^{2} \Tt{q Q} F(\Tt{X}, \Tt{V})\\\label{eq:vec_v}
		&\Td{v}_{n+1}= \Td{v}_0+\Delta t \Tt{q} F(\Tt{X}, \Tt{V})
	\end{align}
\end{subequations}
where $q:=(0, q_{1}, \dots, q_{M})\in \RR^{1 \times (M+1)}$ and $\Tt{q}:=q\otimes \Tt{I}_{d}\in \RR^{d\times d(M+1)}$.
This is the collocation problem for second-order IVPs that our SDC method will solve iteratively.

\subsection{Velocity--Verlet scheme}
We use velocity--Verlet integration~\cite{Verlet1967} as the low-order base method for the SDC iteration for a second-order IVP. 
Applying velocity--Verlet to~\eqref{eq:ode_system} with time steps $\tau_{0}, \dots, \tau_{M}$ gives
\begin{subequations}\label{eq:VV}
\begin{align}
    &\Td{x}_{m+1}=\Td{x}_{m}+\Delta \tau_{m+1}\left(\Td{v}_{m}+\frac{\Delta \tau_{m+1}}{2}\Td{f}_{m}\right),\label{eq:vV_pos}\\
    &\Td{v}_{m+1}=\Td{v}_{m}+\frac{\Delta \tau_{m+1}}{2}(\Td{f}_{m}+\Td{f}_{m+1})\label{eq:vV_vel}
\end{align}
\end{subequations}
where $\Delta \tau_{m+1}=\tau_{m+1}-\tau_{m},\ m=0,\dots, M-1$.
To obtain a matrix formulation for~\eqref{eq:vV_pos}, we convert it into
\begin{subequations}\label{eq:vV_rec}
\begin{align}
    &\Td{x}_{m+1}=\Td{x}_{0}+\sum_{l=1}^{m+1}\Delta \tau_{l}\Td{v}_{l-1}+\frac{1}{2}\sum_{l=1}^{m+1}(\Delta \tau_{l})^{2}\Td{f}_{l-1},\label{eq:vV_rec_pos}\\
    &\Td{v}_{m+1}=\Td{v}_{0}+\frac{1}{2}\sum_{l=1}^{m+1}\Delta \tau_{l}(\Td{f}_{l-1}+\Td{f}_{l}).\label{eq:vV_rec_vel}
\end{align}
\end{subequations}
These equations can be rearranged into vector form by defining
\begin{equation*}\label{eq:QeQI}
Q_{\textrm{E}}:=\frac{1}{\Delta t}\left(
\begin{matrix}
0 & 0 & 0 & \dots & 0\\
\Delta \tau_{1}& 0 & 0& \dots &0\\
\Delta \tau_{1} & \Delta \tau_{2}& 0& \dots& 0\\
\vdots& \vdots & \ddots & \ddots &\vdots\\
\Delta \tau_{1}& \Delta \tau_{2}& \dots& \Delta \tau_{M}& 0
\end{matrix}
\right), \    
Q_{\mathrm{I}}:=\frac{1}{\Delta t}\left(
\begin{matrix}
0 & 0 &0 & \dots & 0\\
0 &\Delta \tau_{1}& 0 & \dots& 0 \\
0 &\Delta \tau_{1} & \Delta \tau_{2}& \dots& 0\\
\vdots& \vdots & \ddots & \ddots &\vdots\\
0&\Delta \tau_{1}& \Delta \tau_{2}& \dots& \Delta \tau_{M}
\end{matrix}
\right)
\end{equation*}
and 
\begin{equation}\label{eq:Qt}
Q_{\textrm{T}}:=\frac{1}{2}(Q_{\textrm{E}}+Q_{\mathrm{I}})\in \RR^{(M+1)\times (M+1)}.
\end{equation}
Then,~\eqref{eq:vV_rec} becomes
\begin{subequations}
\begin{align}
    &\Tt{X}=\Tt{X}_{0}+\Delta t\Tt{Q}_{\textrm{E}}\Tt{V}+\frac{\Delta t^{2}}{2}(\Tt{Q}_{\textrm{E}}\circ\Tt{Q}_{\textrm{E}})F(\Tt{X}, \Tt{V}),\label{eq:vV_vec_pos}\\
    &\Tt{V}=\Tt{V}_{0}+\Delta t\Tt{Q}_{\textrm{T}}F(\Tt{X},\Tt{V})\label{eq:vV_vec_vel}
\end{align}
\end{subequations}
with $\circ$ denoting the Hadamard product (element--wise product of two matrices). 
We substitute the expression for $\Tt{V}$ from~\eqref{eq:vV_vec_vel} into~\eqref{eq:vV_vec_pos} so that
\begin{equation*}
    \Tt{X}=\Tt{X}_{0}+\Delta t\Tt{Q}_{\textrm{E}}\Tt{V}_{0}+\Delta t^{2}\Tt{Q}_{\textrm{x}}F(\Tt{X}, \Tt{V})
\end{equation*}
where
\begin{equation}\label{eq:Qx}
\Tt{Q}_{\textrm{x}}:=Q_{\textrm{x}}\otimes \Tt{I}_{d} \mbox{ with } 
Q_{\textrm{x}}:=Q_{\textrm{E}}Q_{\textrm{T}}+\frac{1}{2}(Q_{\textrm{E}}\circ Q_{\textrm{E}}).
\end{equation}
In order to combine both equations into a compact form based on $\Tt{U}$, we set
\begin{equation*}
\Tt{C}_{\textrm{vv}}:=\left(
\begin{matrix}
	\Tt{I}_{d(M+1)} & \Delta t\Tt{Q}_{\textrm{E}}\\
	\Tt{O} & \Tt{I}_{d(M+1)}
\end{matrix}
\right),\ 
\Tt{Q}_{\textrm{vv}}:=\left(
\begin{matrix}
	\Delta t\Tt{Q}_{\textrm{x}} & \Tt{O}\\
	\Tt{O} & \Tt{Q}_{\textrm{T}}
\end{matrix}
\right)	
\end{equation*}
Finally, the matrix representation of the velocity--Verlet scheme is 
\begin{equation*}
    \Tt{U}=\Tt{C}_{\textrm{vv}}\Tt{U}_{0}+\Delta t \Tt{Q}_{\textrm{vv}}\Tt{F}(\Tt{U})
\end{equation*}
or 
\begin{equation}
    \label{eq:m_vv}
    \Tt{M}_{\textrm{vv}}(\Tt{U}):= \Tt{U} - \Delta t \Tt{Q}_{\textrm{vv}} \Tt{F}(\Tt{U}) = \Tt{C}_{\textrm{vv}}\Tt{U}_{0}.
\end{equation}
We will use $\Tt{M}_{\textrm{vv}}(\cdot)$ as preconditioner for the Picard iteration to solve~\eqref{eq:coll_matrix_final}.
As shown below, $\Tt{M}_{\textrm{vv}}(\cdot)$ is easy to invert by a ``sweep'' with velocity Verlet through all the nodes.

\subsection{Spectral deferred corrections (SDC)}
Applying a Richardson iteration~\cite{Kelley1995} to~\eqref{eq:coll_matrix_final} gives 
\begin{equation}\label{eq:Picard}
    \Tt{U}^{k+1}=(\Tt{I}_{2d(M+1)}-\Tt{M}_{\textrm{coll}})(\Tt{U}^{k})+\Tt{C}_{\textrm{coll}}\Tt{U}_{0}=\Tt{C}_{\textrm{coll}}\Tt{U}_{0}+\Delta t\Tt{Q}_{\textrm{coll}}\Tt{F}(\Tt{U}^{k})
\end{equation}
where $k=0, \dots, K$ is the iteration index and we use $\Tt{U}^{0}$ to start the iteration.
\begin{proposition}\label{co:Picard}
Let $\Td{f}$ be a Lipschitz continuous function with Lipschitz constant $L$ and $\Delta t$ sufficiently small so that $\Delta t L\|\Tt{Q}_{\textrm{coll}}\|<1$ . 
Then,~\eqref{eq:Picard} converges to the collocation solution for all starting values $\Tt{U}_{0}$.
\end{proposition}
\begin{proof}
	Subtracting~\eqref{eq:Picard} for $k+1$ and $k$ yields 
	\begin{equation*}
		\Tt{U}^{k+1}-\Tt{U}^{k}=\Delta t \Tt{Q}_{\textrm{coll}}(\Tt{F}(\Tt{U}^{k})-\Tt{F}(\Tt{U}^{k-1})).
	\end{equation*}
	Applying a norm and using Lipschitz continuity gives us
	\begin{equation*}
		\|\Tt{U}^{k+1}-\Tt{U}^{k}\|\leq \Delta t L\|\Tt{Q}_{\textrm{coll}}\|\|\Tt{U}^{k}-\Tt{U}^{k-1}\|.
	\end{equation*}
	Since $\Delta t L\|\Tt{Q}_{\textrm{coll}}\|<1$, the iteration converges~\cite[p. 1-10]{Praveen2018}.
\end{proof}
Often, the Picard iteration converges only for an impractically small time step. 
To improve convergence, we use $\Tt{M}_{\textrm{vv}}$ as a preconditoner~\cite{Kelley1995}, leading to
\begin{equation}\label{eq:pre_picard}
    \Tt{M}_{\textrm{vv}}(\Tt{U}^{k+1})=(\Tt{M}_{\textrm{vv}}-\Tt{M}_{\textrm{coll}})(\Tt{U}^{k})+\Tt{C}_{\textrm{coll}}\Tt{U}_{0}.
\end{equation}
Each iteration requires solving a linear or non--linear system of equations, depending on the right--hand function $\Td{f}$ in~\eqref{eq:ode_system}, to invert $\Tt{M}_{\textrm{vv}}$.
However, the structure of $\Tt{M}_{\textrm{vv}}$ allows to be done by ``sweeping'' through the quadrature nodes seqentially.
Using~\eqref{eq:m_vv} and~\eqref{eq:coll_matrix} we obtain the operator form of the SDC iteration for second-order equations
\begin{align}
    &(\Tt{I}_{2d(M+1)}-\Delta t\Tt{Q}_{\textrm{vv}}\Tt{F})(\Tt{U}^{k+1})=\Delta t(\Tt{Q}_{\textrm{coll}}-\Tt{Q}_{\textrm{vv}})\Tt{F}(\Tt{U}^{k})+\Tt{C}_{\textrm{coll}}\Tt{U}_{0}\label{eq:SDC_iteration}
\end{align}
for $k=0,\dots, K$. 
\begin{remark}\label{remark:starting_values}
Typically, the starting value $\Tt{U}^0$ for iteration~\eqref{eq:SDC_iteration} will be generated from the initial value $\Tt{U}_0$ at the beginning of the time step, either by setting $\Tt{U}^0 = \Tt{U}_0$ or by an initial sweep of the velocity Verlet base method to solve
\begin{equation}
    \Tt{M}_{\textrm{vv}}(\Tt{U}^{0}) = \Tt{C}_{\textrm{vv}}\Tt{U}_{0}.
\end{equation}
Our theoretical analysis in Subsection~\ref{subsection:theory} does not make any assumptions about how $\Tt{U}^0$ is generated.
Since the aim of the numerical examples in Subsection~\ref{subsection:numerical_examples} is to validate the theory, we always initialize $\Tt{U}^0$ with random values.
Even a simple copy of $\Tt{U}_0$ was found to lead to convergence orders that are better than what the theory guarantees in some cases.
For the comparison of computational efficiency in Section~\ref{sec:efficiency} we use $\Tt{U}^0 = \Tt{U}_0$.
\end{remark}
For analysis, it will be helpful to split the equations for position and velocity
\begin{subequations}
    \label{eq:SDC_vec}
    \begin{align}
 \Tt{X}^{k+1}-\Delta t^{2}\Tt{Q}_{\textrm{x}}F(\Tt{X}^{k+1}, \Tt{V}^{k+1}) &=\Tt{X}_{0}+\Delta t\Tt{Q}\Tt{V}_{0}+\Delta t^{2}(\Tt{QQ}-\Tt{Q}_{\textrm{x}})F(\Tt{X}^{k}, \Tt{V}^{k}),\label{eq:SDC_vec_pos}\\
    \Tt{V}^{k+1}-\Delta t\Tt{Q}_{\textrm{T}}F(\Tt{X}^{k+1}, \Tt{V}^{k+1}) &=\Tt{V}_{0}+\Delta t(\Tt{Q}-\Tt{Q}_{\textrm{T}})F(\Tt{X}^{k}, \Tt{V}^{k})\label{eq:SDC_vec_vel}.
\end{align}
\end{subequations}
Using definitions~\eqref{eq:Qt} and~\eqref{eq:Qx} we obtain the sweep formulation 
\begin{subequations}
    \label{eq:SDC_comp-wise}
\begin{align}
    \Td{x}_{m+1}^{k+1}=\Td{x}_{0}+\Delta t\sum_{l=0}^{M}q_{m+1,l}\Td{v}_{0}+\Delta t^{2}\sum_{l=0}^{M}q_{m+1,l}^{\textrm{x}}\left(\Td{f}_{l}^{k+1}-\Td{f}_{l}^{k}\right)+\Delta t^{2}\sum_{l=0}^{M}qq_{m+1,l}\Td{f}_{l}^{k},\\
    \Td{v}_{m+1}^{k+1}=\Td{v}_{0}+\Delta t\sum_{l=0}^{M}q_{m+1,l}^{\textrm{T}}\left(\Td{f}_{l}^{k+1}-\Td{f}_{l}^{k}\right)+\Delta t\sum_{l=0}^{M}q_{m+1,l}\Td{f}_{l}^{k}
\end{align}
\end{subequations}
where $m=0, \dots, M$ and $k=0,\dots, K, \ f_{l}^{k}:=f(x_{l}^{k}, v_{l}^{k})$ and $(q_{m,l}^{\textrm{x}})_{m,l=0,\dots, M}$ and $(q_{m,l}^{\textrm{T}})_{m,l=0,\dots, M}$ are the entries of $Q_{\textrm{x}}$ and $Q_{\textrm{T}}$.
By taking the difference between~\eqref{eq:SDC_comp-wise} for $m+1$ and $m$ and exploiting that $Q_{\textrm{T}}$ is lower diagonal and $Q_{\textrm{x}}$ strictly lower diagonal, we get
\begin{subequations}
\label{eq:nn_SDC}
\begin{gather}
    \Td{x}_{m+1}^{k+1}=\Td{x}_{m}^{k+1}+\Delta \tau_{m+1}v_{0}+\Delta t^{2}\sum_{l=0}^{m}s_{m+1,l}^{\textrm{x}}\left(\Td{f}_{l}^{k+1}-\Td{f}_{l}^{k}\right)+\Delta t^{2}\sum_{l=0}^{M} sq_{m+1,l}\Td{f}_{l}^{k},\\
   \Td{v}_{m+1}^{k+1}=\Td{v}_{m}^{k+1}+\frac{\Delta \tau_{m+1}}{2}\left(\Td{f}_{m+1}^{k+1}-\Td{f}_{m+1}^{k}\right)+\frac{\Delta \tau_{m+1}}{2}\left(\Td{f}_{m}^{k+1}-\Td{f}_{m}^{k}\right)+\Delta t\sum_{l=0}^{M}s_{m+1,l}\Td{f}_{l}^{k},
\end{gather}
\end{subequations}
for $m=0,\dots, M-1$ and $k=0, \dots, K.$ Here, 
\begin{equation*}
s_{m,j}:=q_{m,j}-q_{m-1,j}, \ s_{m,j}^{\textrm{x}}:=q_{m,l}^{\textrm{x}}-q_{m-1,l}^{\textrm{x}},\ sq_{m,l}:=qq_{m,l}-qq_{m-1,l}	
\end{equation*}
with $m, j=1, \dots, M.$
The factor in front of $\Td{v}_{0}$ is due to  $\Delta t\sum_{j=0}^M s_{m,j} = \Delta \tau_m$.
Since~\eqref{eq:nn_SDC} is a sweep through the quadrature nodes using a velocity-Verlet method with some additional terms on the right-hand side, implementation is straightforward.
\begin{remark}
If $f$ does not depend on $v$,~\eqref{eq:nn_SDC} is a fully explicit SDC iteration.
\end{remark}

\section{Stability}\label{sec:stability}
Similar to how the Dahlquist equation is used to study stability for first-order problems, we use the damped harmonic oscillator with unit mass
\begin{subequations}\label{eq:harmonic_oscillator}
\begin{align}
    &\dot{x}(t)=\Td{v}(t),\\
    &\dot{v}(t)=\Td{f}(\Td{x}(t), \Td{v}(t)) :=-\kappa \Td{x}(t)-\mu \Td{v}(t)
\end{align}
\end{subequations}
as the test problem to study the stability of second-order SDC.
Here, $\kappa$ is the spring constant and $\mu$ the friction coefficient.
Assuming that $\Tt{M}_{\textrm{vv}} = \Tt{I}_{2(M+1)}-\Tt{Q}_{\textrm{vv}}\Tt{F}$ is invertible, iteration~\eqref{eq:SDC_iteration} becomes
\begin{equation}\label{eq:linear_SDC}
	\Tt{U}^{k+1}=\Tt{K}_{\textrm{sdc}}\Tt{U}^{k}+(\Tt{I}_{2(M+1)}-\Delta t\Tt{Q}_{\textrm{vv}}\Tt{F})^{-1}\Tt{C}_{\textrm{coll}}\Tt{U}_{0},
\end{equation}
where $\Tt{K}_{\textrm{sdc}}:=(\Tt{I}_{2(M+1)}-\Delta t\Tt{Q}_{\textrm{vv}}\Tt{F})^{-1}(\Delta t\Tt{Q}_{\textrm{coll}}-\Delta t\Tt{Q}_{\textrm{vv}})\Tt{F}$.
For fixed $M$ and choice of quadrature nodes, the iteration matrix $\Tt{K}_{\text{sdc}}$ depends only on $\Delta t \kappa$, $\Delta t \mu$.
The iteration matrix is similar to the one in first-order problems~\cite{HuangEtAl2006}.

There are two different but related issues regarding stability of SDC: convergence of the SDC iteration for a single time step as $K \to \infty$ and boundedness of the sequence of approximations $x_n$, $v_n$ generated by subsequent applications of SDC as $n \to \infty$.
\begin{proposition}\label{prop:conv}
	The sequence $\{\Tt{U}^{k}\}$ generated by~\eqref{eq:linear_SDC} converges for any $\Tt{U^{0}}$ and starting values $\Tt{U}_{0}$ if and only if 
	\begin{equation*}
		\rho(\Tt{K}_{\textrm{sdc}}):=\max_{\lambda\in \mathrm{spec}(\Tt{K}_{\textrm{sdc}})}|\lambda|<1.	
	\end{equation*}
\end{proposition}
\begin{proof}
	The proof works along the lines of the proof of~\cite[Theorem 2.16]{hackbusch1994iterative}.
\end{proof}
If we identify a pair of positive parameters $\Delta t ( \kappa, \mu)$ with a point in the positive quadrant $\mathbb{R}^2_{+}$, we can define the \emph{convergence domain} of SDC as
\begin{equation}
    \Omega_{\text{conv}} := \left\{ (\Delta t \kappa, \Delta t \mu) \in \mathbb{R}^2_{+} : \rho(\Tt{K}_{\text{sdc}}) < 1 \right\}.
\end{equation}
For a set of parameters inside $\Omega_{\text{conv}}$, SDC will converge to the solution of the collocation problem~\eqref{eq:coll_main} as $k \to \infty$.

To assess stability as $n \to \infty$, we derive the stability function of SDC.
Using induction and~\eqref{eq:SDC_iteration} we find that
\begin{equation*}
	\Tt{U}^{k+1}=\Tt{K}_{\textrm{sdc}}^{k+1}\Tt{U}_{0}+\sum_{j=0}^{k}\Tt{K}_{\textrm{sdc}}^{j}\Tt{M}_{\textrm{vv}}^{-1}\Tt{C}_{\textrm{coll}}\Tt{U}_{0}.
\end{equation*}
Using the geometric series formula we obtain
\begin{equation*}
	\Tt{U}^{k+1}=\Tt{K}_{\textrm{sdc}}^{k+1}\Tt{U}_{0}+(\Tt{I}_{2(M+1)}-\Tt{K}_{\textrm{sdc}}^{k+1})(\Tt{I}_{2(M+1)}-\Tt{K}_{\textrm{sdc}})^{-1}\Tt{M}_{\textrm{vv}}^{-1}\Tt{C}_{\textrm{coll}}\Tt{U}_{0}.
\end{equation*}
This can be slightly simplified to
\begin{equation}\label{eq:SDC_U0}
	 \Tt{U}^{k+1}=\Tt{P}^{k+1}_{\textrm{sdc}}\Tt{U}_{0}
\end{equation}
where
\begin{equation*}
\Tt{P}^{k}_{\textrm{sdc}}:=\Tt{K}_{\textrm{sdc}}^{k}+(\Tt{I}_{2(M+1)}-\Tt{K}_{\textrm{sdc}}^{k})(\Tt{I}_{2(M+1)}-\Tt{K}_{\textrm{sdc}})^{-1}\Tt{M}_{\textrm{vv}}^{-1}\Tt{C}_{\textrm{coll}}.	
\end{equation*}
The final quadrature step~\eqref{eq:up_vec} can also be written in matrix form 
\begin{equation}\label{eq:up_mat}
    \left(\begin{matrix}
    \Td{x}_{n+1}\\
    \Td{v}_{n+1}
    \end{matrix}\right)=\left(\begin{matrix}
    1 & \Delta t\\
    0 & 1
    \end{matrix}\right)\left(\begin{matrix}
    \Td{x}_{0}\\
    \Td{v}_{0}
    \end{matrix}\right)+\left(\begin{matrix}
    \Delta t^{2}\Tt{qQ} & \Tt{0}\\
    \Tt{0} & \Delta t \Tt{q}
    \end{matrix}\right)\Tt{F}\Tt{U}^{k+1}.
\end{equation}
Inserting the expression for $\Tt{U}^{k+1}$ from~\eqref{eq:SDC_U0} into~\eqref{eq:up_mat} yields
\begin{equation}\label{eq:up_mat_SDC}
    \left(\begin{matrix}
    \Td{x}_{n+1}\\
    \Td{v}_{n+1}
    \end{matrix}\right)=\left(\begin{matrix}
    1 & \Delta t\\
    0 & 1
    \end{matrix}\right)\left(\begin{matrix}
    \Td{x}_{0}\\
    \Td{v}_{0}
    \end{matrix}\right)+\left(\begin{matrix}
    \Delta t^{2}\Tt{qQ} & \Tt{0}\\
    \Tt{0} & \Delta t \Tt{q}
    \end{matrix}\right)\Tt{F}\Tt{P}^{k+1}_{\textrm{sdc}}\Tt{U}_{0}
\end{equation}
or
\begin{equation}
    \left(\begin{matrix}
    \Td{x}_{n+1}\\
    \Td{v}_{n+1}
    \end{matrix}\right)=\left(\begin{matrix}
    1 & \Delta t\\
    0 & 1
    \end{matrix}\right)
    \left(\begin{matrix}
    \Td{x}_{0}\\
    \Td{v}_{0}
    \end{matrix}\right)+\left(\begin{matrix}
    \Delta t^{2}\Tt{qQ} & \Tt{0}\\
    \Tt{0} & \Delta t \Tt{q}
    \end{matrix}\right)\Tt{F}\Tt{P}^{k+1}_{\textrm{sdc}}\bar{\mathbf{1}}\left(\begin{matrix}
    \Td{x}_{0}\\
    \Td{v}_{0}
    \end{matrix}\right)
\end{equation}
where $\bar{\mathbf{1}}=\left(\begin{matrix}
	\mathbf{1} & \Tt{0}\\
	\Tt{0} & \mathbf{1}
\end{matrix}\right)\in\mathbb{R}^{2(M+1)\times 2}$ with $\mathbf{1}=(1,1, \dots ,1)^{T}\in\mathbb{R}^{M+1}.$
A full-step of SDC from $t_n$ to $t_{n+1}$ for the damped harmonic oscillator therefore becomes
\begin{equation}\label{eq:full-step_SDC}
    \left(\begin{matrix}
    \Td{x}_{n+1}\\
    \Td{v}_{n+1}
    \end{matrix}\right)=\left(\left(\begin{matrix}
    1 & \Delta t\\
    0 & 1
    \end{matrix}\right)
   +\left(\begin{matrix}
    \Delta t^{2}\Tt{qQ} & \Tt{0}\\
    \Tt{0} & \Delta t \Tt{q}
    \end{matrix}\right)\Tt{F}\Tt{P}^{k+1}_{\textrm{sdc}}\bar{\mathbf{1}}\right)\left(\begin{matrix}
    \Td{x}_{0}\\
    \Td{v}_{0}
    \end{matrix}\right).
\end{equation}
and the stability function of SDC iteration is
\begin{equation}\label{eq:stab_domain}
    R(\Delta t \kappa, \Delta t \mu)=\left(\begin{matrix}
    1 & \Delta t\\
    0 & 1
    \end{matrix}\right)+\left(\begin{matrix}
    \Delta t^{2}\Tt{qQ} & \Tt{0}\\
    \Tt{0} & \Delta t\Tt{q}
    \end{matrix}\right)\Tt{F} \Tt{P}^{k+1}_{\textrm{sdc}}\bar{\mathbf{1}}.
\end{equation}
Stability in the sense that $x_n$ and $v_n$ remain bounded as $n \to \infty$ is then ensured if $(\Delta t \kappa, \Delta t \mu)$ is an element of the stability domain
\begin{equation}
     \Omega_{\text{stab}} := \left\{ (\Delta t \kappa, \Delta t \mu) \in \mathbb{R}^2_{+} : \rho(R(\Delta t \kappa, \Delta t \mu)) < 1 \right\}.
\end{equation}

Figure~\ref{fig:stab1} shows the stability domain after $K=50$ iterations (upper left) and the convergence domain (upper right) for $M=3$ Gauss-Legendre nodes.
Note how the boundaries of the convergence and stability domain coincide.
In general, the stability domains grows as $M$ increases.
While this is not documented here, readers can generate stability domains for larger values of $M$ using the provided code.
\begin{figure}[t]
	\centering
	\includegraphics[scale=0.34]{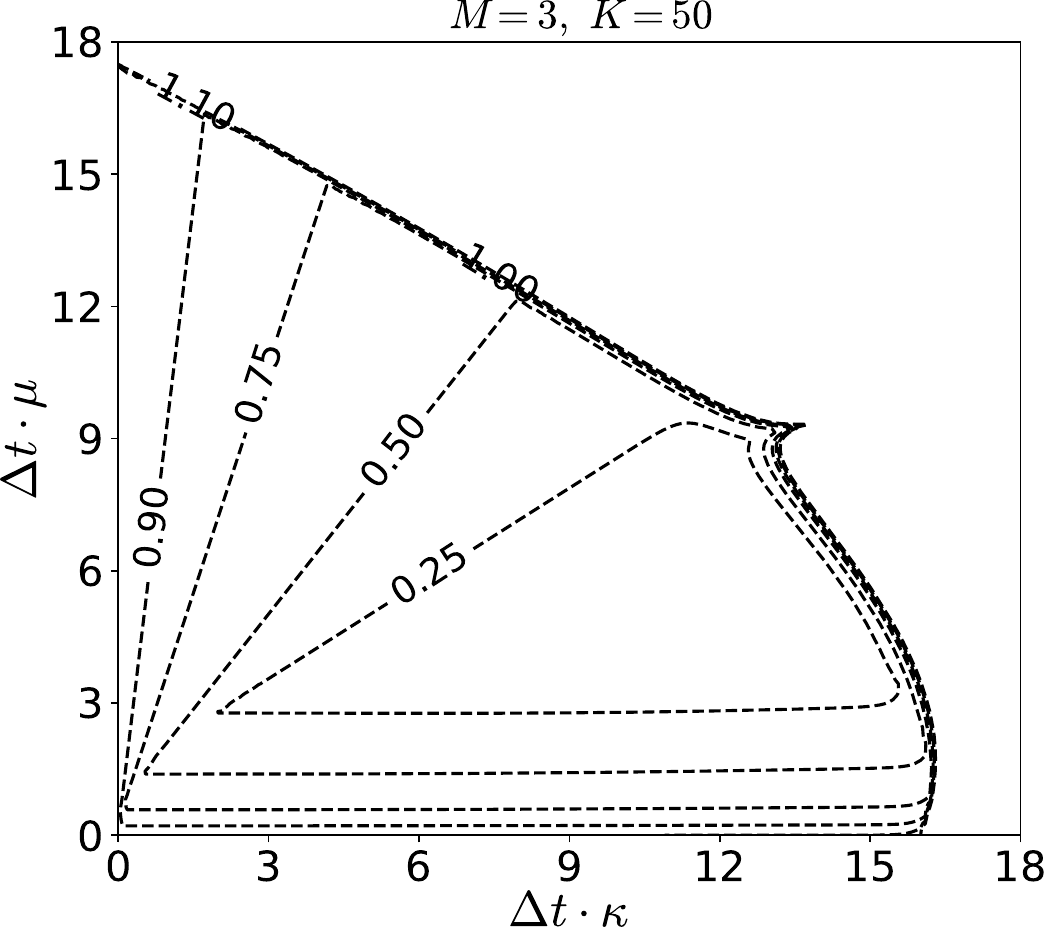}
	\includegraphics[scale=0.34]{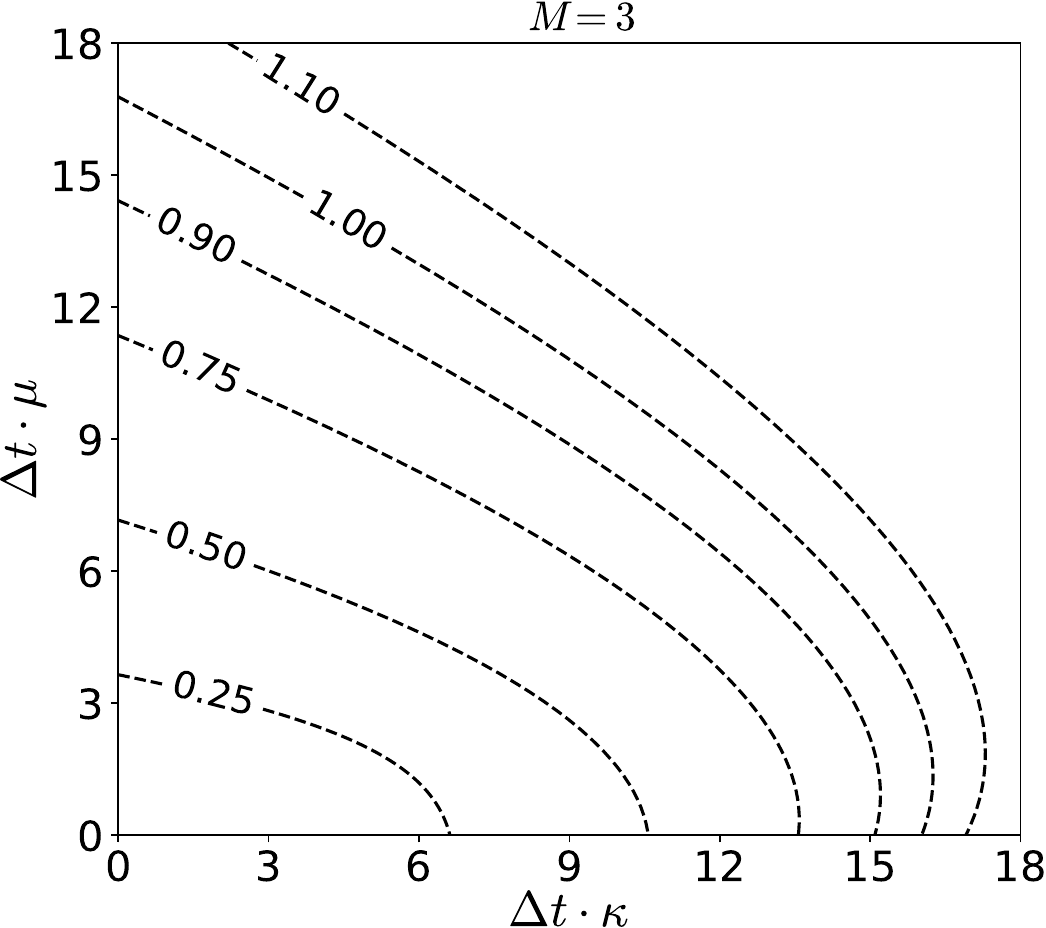}
	 \includegraphics[scale=0.34]{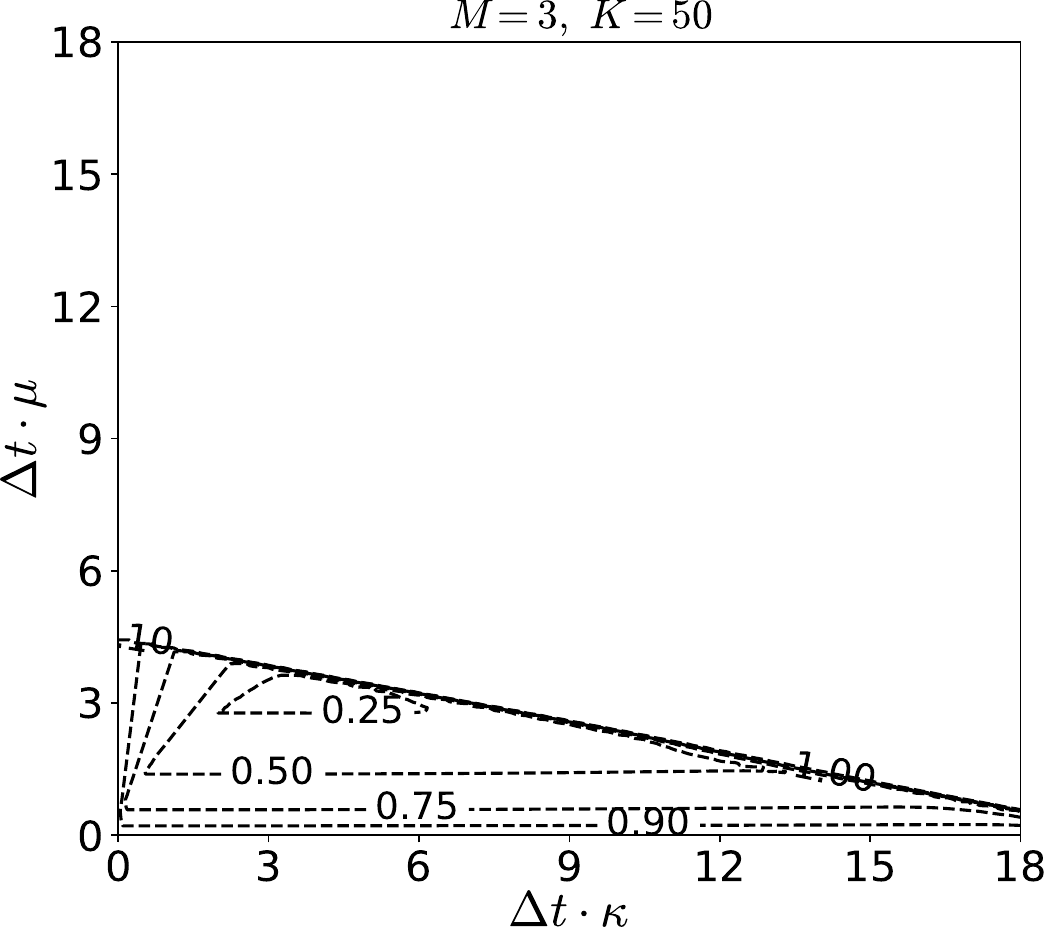}
	\includegraphics[scale=0.34]{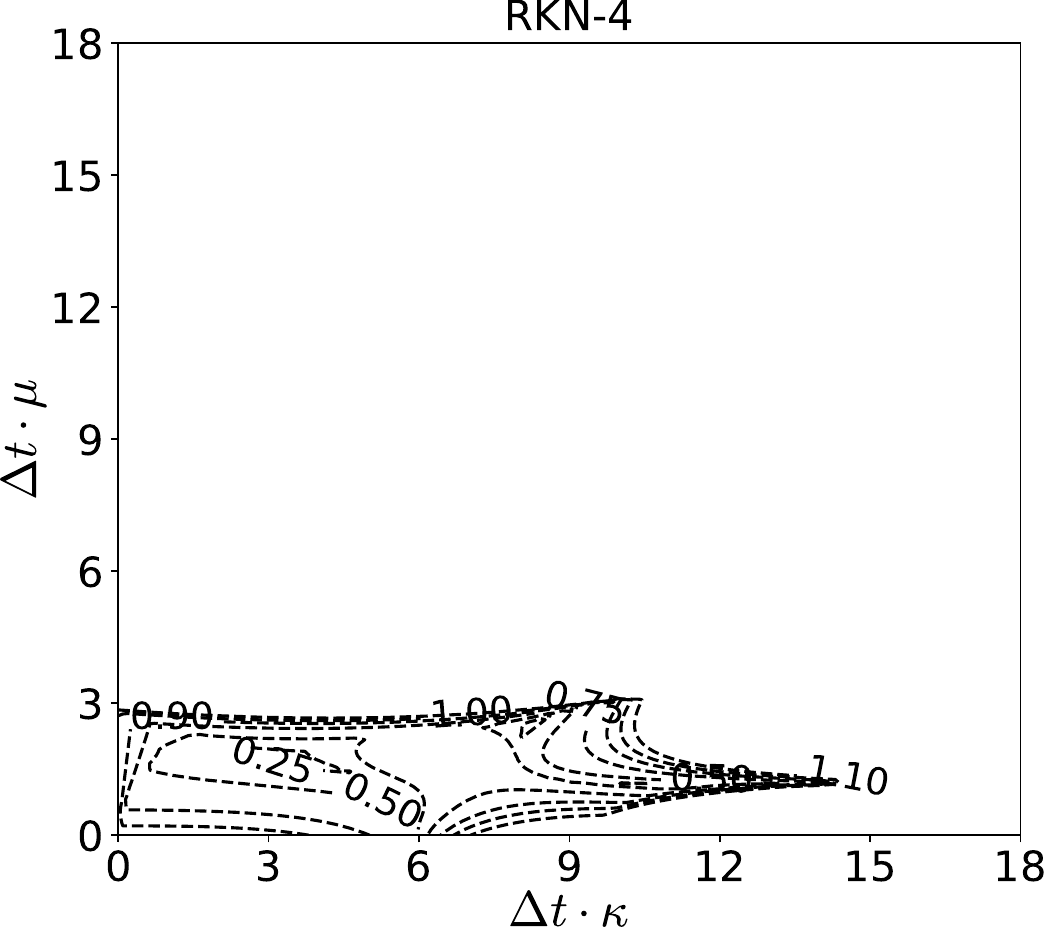}
	\caption{Stability domain for $K=50$ iterations (upper left) and convergence domain (upper right) of SDC with $M=3$ Gauss-Legendre quadrature nodes. Stability domain of Picard iteration with $K=50$ iterations and $M=3$ nodes (lower left) and stability domain of RKN-4 (lower right).}
	\label{fig:stab1}
\end{figure}

Figure~\ref{fig:stab1} also compares the stability/convergence domains of the Picard iteration~\eqref{eq:Picard} (lower left) and RKN-4 (lower right) with the SDC iteration~\eqref{eq:linear_SDC} (upper left) for $K=50$.
For the undamped system with $\mu = 0$, the Picard iteration converges up to around $\Delta t \kappa = 18$ while SDC only converges until $\Delta t \kappa = 16$, although neither method will provide accuracy for such a low resolution.
However, once damping is added to the system SDC converges for a much larger range of parameters.
In particular, SDC converges for the stiff case with very strong damping while Picard does not.
The stability domains changes when the number or type of quadrature nodes changes.
The reader can use the provided code to generate stability domains for other choices.

Figure~\ref{fig:stab3} illustrates how the stability domain of SDC changes with the number of iterations $K$.
For $K=1$, the stability domain is noticeably smaller than the convergence domain.
Surprisingly, for $K=2$, the whole shown range of parameters becomes stable -- at the moment, we have no theoretical explanation.
\rev{A preliminary parameter search suggests that Lobatto nodes in particular often produce methods that remain stable for extremely strong damping (up to $\Delta t \mu \approx 100$) but we were not able to identify a robust heuristic for this behaviour.} For $K=3$ iterations, parts of the shown parameter range are unstable again but the stability domain is still significantly larger than for $K=1$.
Increasing to $K=4$ iterations increase the stability domain into the direction of stronger damping but slightly decreases it in the direction of a larger spring constant.
However, in all cases the stability domain of SDC is much larger than that of Picard or RKN-4.
\rev{\paragraph{Stability for the purely oscillatory case}
Table~\ref{table:kappa} shows the maximum stable values for SDC and Picard iteration for $\Delta t \kappa$ along the $x$-axis, that is for the purely oscillatory case with no damping ($\mu = 0$).
Choosing an even number of iterations $K$ seems to be a poor choice for purely oscillatory systems as both methods are either unstable or have very restrictive stability limits.
By contrast, if $K$ is odd, both methods are stable for very large steps with SDC allowing even larger stable time steps than Picard iterations.
At the moment, we cannot offer a hypothesis what causes this very different behaviour for odd and even $K$.
}
\begin{table}
	\centering
	\caption{\rev{Stability limit for $\Delta t \kappa$ for $\mu =0$ (purely oscillatory case with no damping) rounded to the first digit for SDC and Picard iteration (in brackets).}}
	\label{table:kappa}
\begin{tabular}{|| c| c |c | c| c| c||} 
	\hline
	\multicolumn{6}{|c|}{SDC~(Picard)} \\
	\hline
	$K$ &  $M=2$ & $M=3$ & $M=4$ & $M=5$ & $M=6$\\
	\hline
	$1$ & $6.0~(4.7) $ & $7.2~(4.7)$ & $7.8~(4.7)$& $8.4~(4.7)$& $8.6~(4.7)$\\
	$2$ & $0.0~(12.0)$ & $0.0~(0.0)$ & $0.0~(0.0)$ & $0.0~(0.0)$ & $0.0~(0.0)$\\
	$3$ & $0.0~(0.0)$  & $9.6~(7.1)$ & $26.5~(4.0)$ & $35.3~(4.0)$ & $55.1~(4.0)$\\
	$4$ & $11.6~(7.0)$            & $0.2~(0.1)$ & $0.4~(0.2)$ & $0.4~(0.2)$ & $0.6~(0.2)$\\
	\hline
\end{tabular}
	\medskip

\end{table}

\begin{figure}[t]
    \centering
	\includegraphics[scale=0.34]{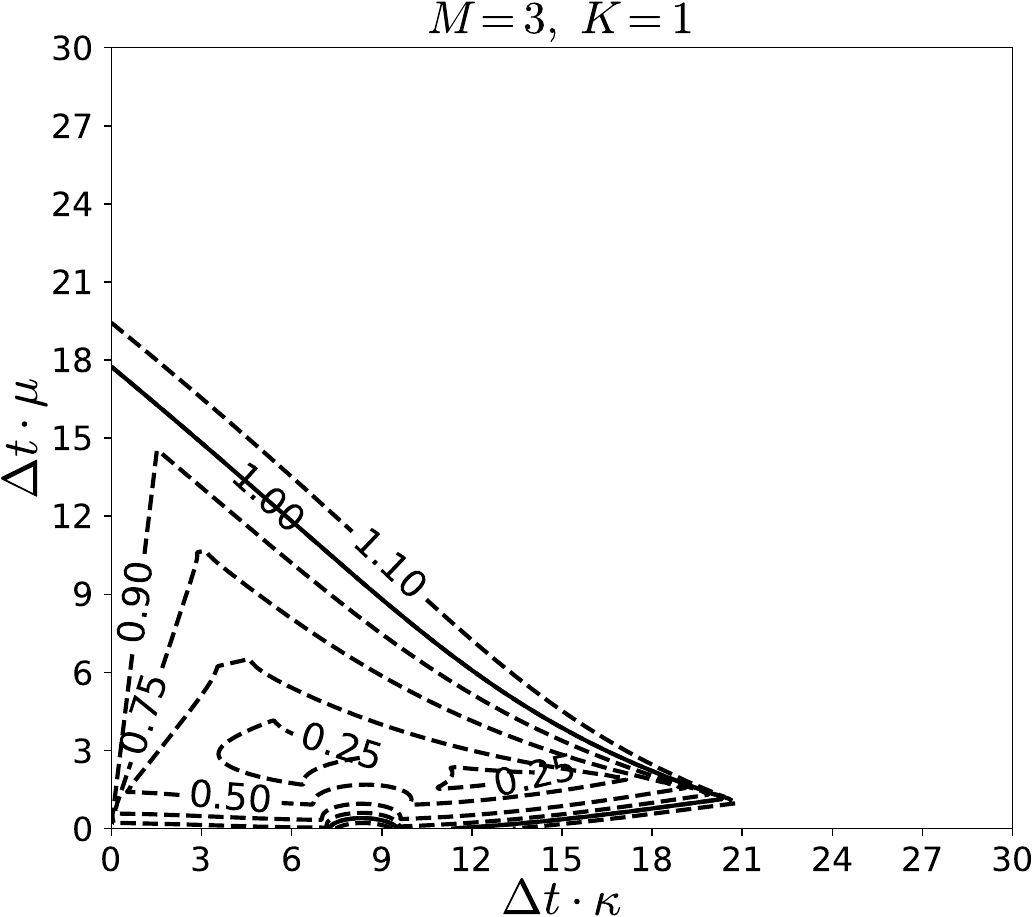}
	\includegraphics[scale=0.34]{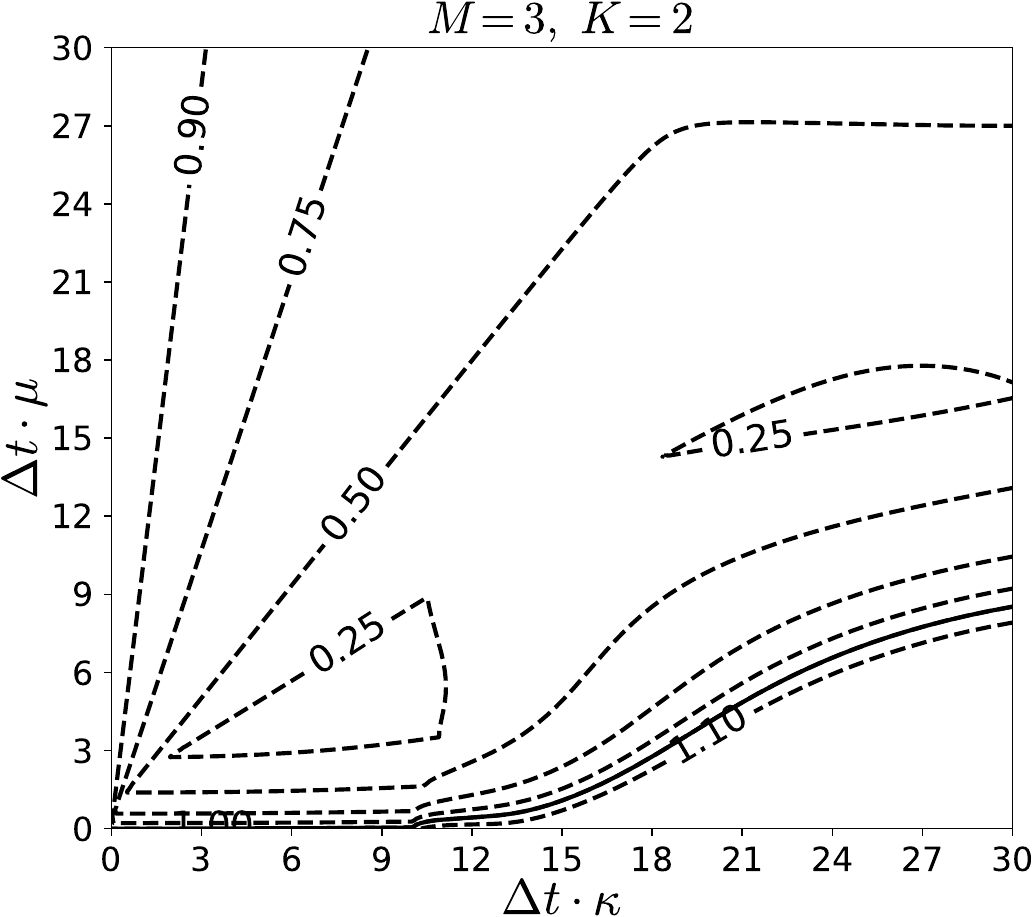}
	\includegraphics[scale=0.34]{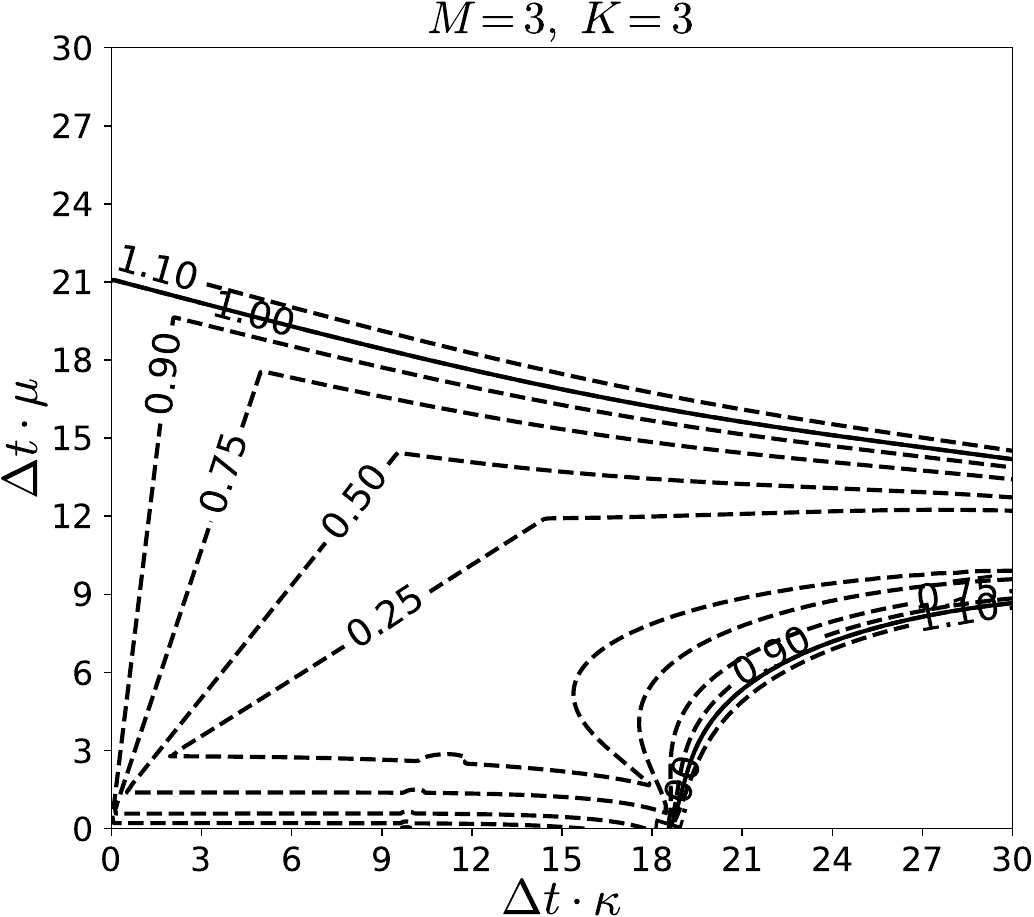}
	\includegraphics[scale=0.34]{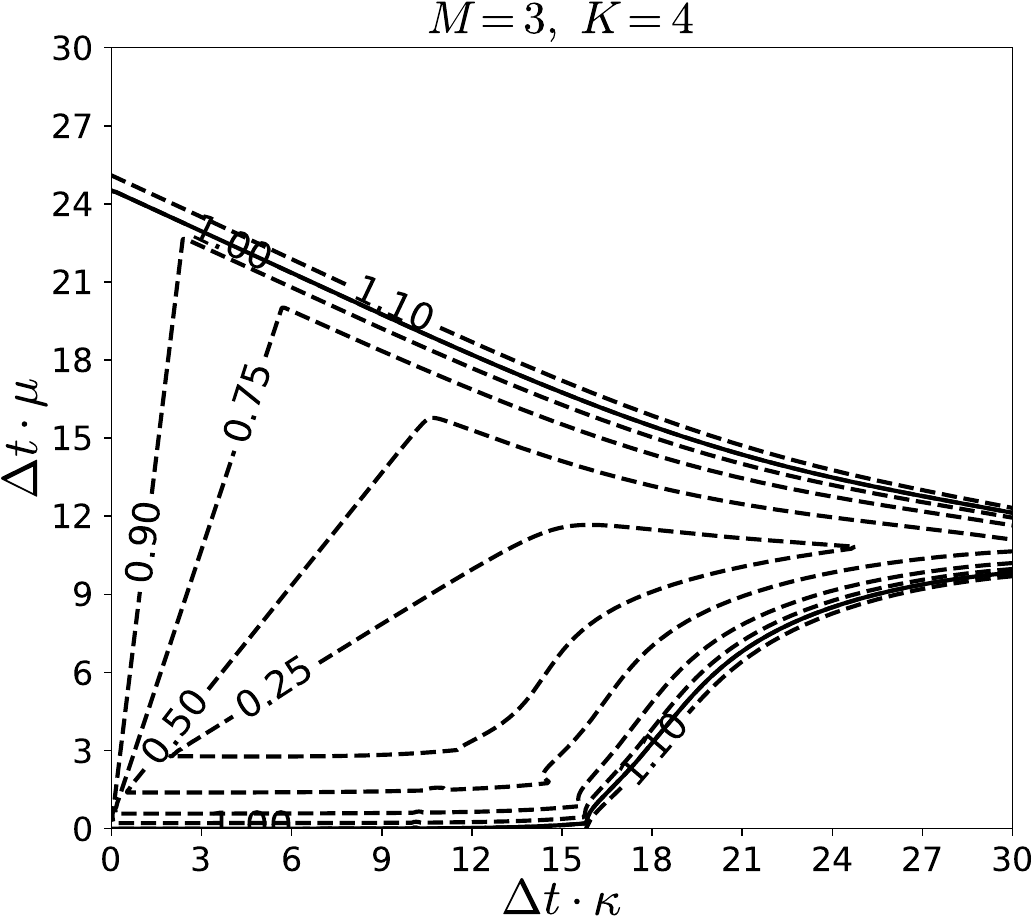}
	\caption{Stability domains of SDC with $M=3$ Gauss-Legendre nodes and $K=1,2,3,4$ iterations.}
	\label{fig:stab3}
\end{figure}
\section{Consistency and convergence order}\label{sec:convergence}
We state and prove our main theoretical result on the convergence rate of SDC for second-order problems in Section~\ref{subsection:theory} and then validate the theory against numerical examples for the special case of a single charged particle in a Penning trap in Section~\ref{subsection:numerical_examples}.

\subsection{Theory}\label{subsection:theory}
The strategy we use to prove the Theorem below follows approaches used for the first-order case~\cite{CausleySeal2019,KremlingSpeck2021}. 
\begin{theorem}\label{the:main}
Consider the initial value problem~\eqref{eq:ode_system} with  a Lipschitz continuous function $f$ with Lipschitz constant $L$.
Let $p$ denote the order of the quadrature rule, and 
assume that $\Td{f}\circ (\Td{x},\Td{v})\in \mathbf{C}^{p}([t_{n}, t_{n+1}])$ and that there exists a positive constant $G$ such that $ \|\frac{d^{p}}{d t^{p}}(\Td{f} \circ (\Td{x},\Td{v}))\|\leq G$. 
\rev{Let $(x(t_{n+1}),\ v(t_{n+1}))$ be the exact solutions to~\eqref{eq:ode_system}} and $(\Td{x}_{n+1}^{k},\ \Td{v}_{n+1}^{k})$ be the approximate solutions to~\eqref{eq:coll_SDC} provided by SDC after $k$ iterations. 
If the step size $\Delta t$ is sufficiently small, then 
\begin{subequations}
\begin{align}
	&|x(t_{n+1})-x^{k}_{n+1}|=\mathcal{O}(G\Delta t^{p+1})+ \mathcal{O}(L^{k+1}\Delta t^{k+k_{0}+2})\\
	&|v(t_{n+1})-v^{k}_{n+1}|=\mathcal{O}(G\Delta t^{p+1})+ \mathcal{O}(L^{k+1}\Delta t^{k+k_{0}+1}),
\end{align}	
\end{subequations}
where $k_{0}$ denotes the approximation order of the base method \rev{used to generate $\Tt{U}^0$, see Remark~\ref{remark:starting_values}}.   
Moreover, if $\Td{f}$ is independent of $\Td{v}$, we have  
\begin{subequations}
\begin{align}
 	&|x(t_{n+1})-x^{k}_{n+1}|=\mathcal{O}(G\Delta t^{p+1})+ \mathcal{O}(L^{k+1}\Delta t^{2k+k_{0}+2})\\
 	&|v(t_{n+1})-v^{k}_{n+1}|=\mathcal{O}(G\Delta t^{p+1})+ \mathcal{O}(L^{k+1}\Delta t^{2k+k_{0}+1}).
\end{align}
\end{subequations}
\end{theorem}
\begin{proof}
\rev{We substitute~\eqref{eq:dis_coll} into~\eqref{eq:vec_v} to find the updates $(x_{n+1}, v_{n+1})$ for the collocation method. Additionally, we determine the SDC method update formula $(x_{n+1}^{k}, v_{n+1}^{k})$ from~\eqref{eq:SDC_vec} by plugging it into~\eqref{eq:dis_coll},} subtract and use Cauchy-Schwarz inequality~\cite[p. 171-177]{Gilbert2005} and Lipschitz continuity to get
\begin{align}\label{eq:ineq_vel}
    |v_{n+1}-v^{k}_{n+1}| &= \Delta t|\Tt{q}(F(\Tt{X},\Tt{V})-F(\Tt{X}^{k}, \Tt{V}^{k}))| \\
    &\leq \Delta t\|\Tt{q}\|\|F(\Tt{X},\Tt{V})-F(\Tt{X}^{k}, \Tt{V}^{k})\| \\
    & \leq \Delta t L\|\Tt{q}\|(\|\Tt{X}-\Tt{X}^{k}\|+\|\Tt{V}-\Tt{V}^{k}\|).  
\end{align}
Using that $\|q\|\leq 1$~\cite{RuprechtSpeck2016} and Theorem~\ref{Th:bound} we find that
\begin{align}
		|v_{n+1}-v^{k}_{n+1}| &\leq \Delta t L\|\Tt{q}\|(\|\Tt{X}-\Tt{X}^{k}\|+\|\Tt{V}-\Tt{V}^{k}\|)\\ &\leq
	\Delta t L^{k+1}(\tilde{C_{1}}\Delta t^{k+k_0+1}+\tilde{C_{2}}\Delta t^{k+k_0}) \\
	 &\leq L^{k+1}(\tilde{C_{1}}+\tilde{C_{2}}\Delta t)\Delta t^{k+k_0+1}
\end{align}
The entries of the $qQ$ satisfy~\cite[p. 208-210]{HairerEtAl1993_nonstiff}
\begin{equation}
\sum_{i=0}^{M}q_{i}q_{i,j}=q_{j}(1-\tau_{j}), \ j=0,\dots,M. 	
\end{equation}
Because $\tau_{j}\leq 1$ for all $j=1, \dots, M$ on the unit interval, it holds that 
\begin{equation}
\|\Tt{q}\Tt{Q}\|=\max_{j=1,\dots , M}|q_{j}(1-\tau_{j})|\leq 1.
\end{equation}
Plugging~\eqref{eq:dis_coll} and~\eqref{eq:SDC_vec} into~\eqref{eq:vec_x}, subtracting and using Cauchy-Schwarz~inequality and Lipschitz continuity gives
\begin{align}\label{eq:ineq_pos}
     |\Td{x}_{n+1}-\Td{x}^{k}_{n+1}| =\Delta t^{2}|\Tt{qQ}(F(\Tt{X},\Tt{V})-F(\Tt{X}^{k}, \Tt{V}^{k}))| \\
     \leq L\Delta t^{2} \|\Tt{qQ}\|(\|\Tt{X}-\Tt{X}^{k}\|+\|\Tt{V}-\Tt{V}^{k}\|). 
\end{align}
Using Theorem~\ref{Th:bound} yields
\begin{align}
		|\Td{x}_{n+1}-\Td{x}^{k}_{n+1}| &\leq \Delta t^{2} L\|\Tt{qQ}\|(\|\Tt{X}-\Tt{X}^{k}\|+\|\Tt{V}-\Tt{V}^{k}\|)\\ &\leq
		\Delta t^{2} L^{k+1}(\tilde{C_{1}}\Delta t^{k+k_0+1}+\tilde{C_{2}}\Delta t^{k+k_0}) \\
		 &\leq L^{k+1}(\tilde{C_{2}}+\tilde{C_{1}}\Delta t)\Delta t^{k+k_0+2}
\end{align}
Assuming that $\Delta t \leq 1$, we have
\begin{equation*}
\tilde{C_{2}}+\tilde{C_{1}}\Delta t\leq \tilde{C_{2}}+\tilde{C_{1}} =: C_{L}	
\end{equation*}
and therefore
\begin{subequations}
\label{eq:nodes_order}
\begin{align}
		&|\Td{x}_{n+1}-\Td{x}^{k}_{n+1}|\leq C_{L}L^{k+1}\Delta t^{k+k_{0}+2},\label{eq:nodes_order_pos}\\
		&|\Td{v}_{n+1}-\Td{v}^{k}_{n+1}|\leq C_{L}L^{k+1}\Delta t^{k+k_{0}+1}.\label{eq:nodes_order_vel}
	\end{align}	
\end{subequations}
Gauss quadrature nodes satisfy the orthogonality condition 
\rev{\begin{equation*}
	\int_{0}^{1}s^{j-1}\prod_{i=1}^{M}(s-\tau_{i}) d s=0, \quad j=1,2, \dots, \xi.
\end{equation*}
~\cite[Theorem~7.9]{HairerEtAl1993_nonstiff}.
Thus}, the following estimates hold
\begin{subequations}\label{eq:coll_order}
\begin{align}
	|\Td{x}(t_{n+1})-\Td{x}_{n+1}|\leq C_{G} G\Delta t^{p+1},\label{eq:coll_order_pos}\\
	|\Td{v}(t_{n+1})-\Td{v}_{n+1}|\leq  C_{G}G\Delta t^{p+1},\label{eq:coll_order_vel}
\end{align}
\end{subequations}
where $p=M+\xi$ and $C_{G}$ is a constant. 
We have $\xi = M$ and $p=2M$ for Legendre nodes, $\xi = M-1$ and $p=2M-1$ for Radau nodes and $\xi = M-2$ and $p=2M-2$ for Lobatto nodes.
Subtracting the analytical solution from the SDC solution at time $t_{n+1}$ and using the triangle inequality along with~\eqref{eq:nodes_order} and~\eqref{eq:coll_order} gives the bound
\begin{align}\label{eq:pos_order}
        |\Td{x}(t_{n+1})-\Td{x}^{k}_{n+1}| \leq |\Td{x}(t_{n+1})-\Td{x}_{n+1}|+|\Td{x}_{n+1}-\Td{x}^{k}_{n+1}| \\
        \leq C_{G} G\Delta t^{p+1}+C_{L}L^{k+1}\Delta t^{k+k_{0}+2}. \nonumber
\end{align}
for the position error and the bound
\begin{align}\label{eq:vel_order}
        |\Td{v}(t_{n+1})-\Td{v}^{k}_{n+1}| \leq |\Td{v}(t_{n+1})-\Td{v}_{n+1}|+|\Td{v}_{n+1}-\Td{v}^{k}_{n+1}| \\
        \leq C_{G} G\Delta t^{p+1}+C_{L}L^{k+1}\Delta t^{k+k_{0}+1}. \nonumber
\end{align}
for the velocity error.
In summary, the local error of second-order SDC satisfies
\begin{subequations}
\begin{align}
		&|\Td{x}(t_{n+1})-\Td{x}^{k}_{n+1}| = \mathcal{O}(G\Delta t^{p+1})+\mathcal{O}(L^{k+1}\Delta t^{k+k_{0}+2}),\\
		 &|\Td{v}(t_{n+1})-\Td{v}^{k}_{n+1}|= \mathcal{O}(G\Delta t^{p+1})+\mathcal{O}(L^{k+1}\Delta t^{k+k_{0}+1}).
\end{align}
\end{subequations}
When $\Td{f}$ is independent of $\Td{v}$, equations~\eqref{eq:ineq_vel} and~\eqref{eq:ineq_pos} become
\begin{equation*}
		|\Td{x}_{n+1}-\Td{x}^{k}_{n+1}|=\Delta t^{2}|\Tt{qQ}(F(\Tt{X})-F(\Tt{X}^{k}))|\leq L\Delta t^{2} \|\Tt{qQ}\|\|\Tt{X}-\Tt{X}^{k}\| \leq L\Delta t^{2} \|\Tt{X}-\Tt{X}^{k}\|
\end{equation*}
and
\begin{equation*}
	|\Td{v}_{n+1}-\Td{v}^{k}_{n+1}|=\Delta t|\Tt{q}(F(\Tt{X})-F(\Tt{X}^{k}))|\leq L\Delta t \|\Tt{q}\|\|\Tt{X}-\Tt{X}^{k}\|\leq L\Delta t\|\Tt{X}-\Tt{X}^{k}\|. 
\end{equation*}
Using triangle inequality,~\eqref{eq:coll_order} and Theorem~\ref{Th:bound} yields
\begin{align}
        \label{eq:estimation_up_pos}
		|\Td{x}(t_{n+1})-\Td{x}^{k}_{n+1}| \leq |\Td{x}(t_{n+1})-\Td{x}_{n+1}|+|\Td{x}_{n+1}-\Td{x}^{k}_{n+1}| \\
		\leq C_{G}G\Delta t^{p+1}+C_{L}L^{k+1}\Delta t^{2k+k_0+2}
\end{align}
and 
\begin{align}
    \label{eq:estimation_up_vel}
		|\Td{v}(t_{n+1})-\Td{v}^{k}_{n+1}| \leq |\Td{v}(t_{n+1})-\Td{v}_{n+1}|+|\Td{v}_{n+1}-\Td{v}^{k}_{n+1}| \\
		\leq  C_{G}G\Delta t^{p+1}+C_{L}L^{k+1}\Delta t^{2k+k_{0}+1}.
\end{align}
Thus we obtain
\begin{subequations}
\begin{align}
		&|\Td{x}(t_{n+1})-\Td{x}^{k}_{n+1}| = \mathcal{O}(G\Delta t^{p+1})+\mathcal{O}(L^{k+1}\Delta t^{2k+k_0+2}),\\
		&|\Td{v}(t_{n+1})-\Td{v}^{k}_{n+1}| = \mathcal{O}(G\Delta t^{p+1})+\mathcal{O}(L^{k+1}\Delta t^{2k+k_{0}+1}).
	\end{align}
\end{subequations}
\end{proof}
A direct consequence of Theorem~\ref{the:main} and~\cite[Definition~2.1]{sommeijer1993explicit} is the following:
\begin{theorem}
    \label{thm:global}
	Let the right-hand side function $f$ in~\eqref{eq:ode_system} satisfy the assumptions of Theorem~\ref{the:main}. Then, the global convergence rate of SDC is $p^{*}:=\min\{p, k+k_0\}$. 
	If $f$ does not depend on $v$, we have $p^{*} =\min\{p, 2k+k_0\}$.
\end{theorem}

\subsection{Numerical examples}\label{subsection:numerical_examples}
We validate our convergence analysis for the Penning trap benchmark~\cite{Penning1936}. 
The equations of motion are the Lorentz equations
\begin{subequations}
\label{eq:motion}
\begin{align}
    &\dot{x}(t)=\Td{v}(t),\\
    &\dot{v}(t)=\Td{f}(\Td{x}(t),\Td{v}(t)):=\alpha[\Tt{E}(\Td{x}(t))+\Td{v}(t) \times \Tt{B}(\Td{x}(t))]
\end{align}
\end{subequations}
with a constant magnetic field $\Tt{B}=\frac{\omega_{B}}{\alpha}\cdot \hat{\Tt{e}}_{z}$ along the $z-$axis with frequency $\omega_{B}$.
Let $\alpha=\frac{q}{m}$ denote that particle's charge--to--mass ratio so that
\begin{equation}\label{eq:B}
    \Td{v}\times \Tt{B}=\frac{\omega_{B}}{\alpha}\left(
    \begin{matrix}
    0 & 1 & 0\\
    -1 & 0 & 0\\
    0 & 0 &0 
    \end{matrix}
    \right)\Td{v}.
\end{equation}
The electric field with frequency $\omega_{E}$ is given by 
\begin{equation}\label{eq:E}
    \Tt{E}(\Td{x})=-\epsilon\frac{\omega_{E}^{2}}{\alpha}\left(
    \begin{matrix}
    1 & 0 & 0\\
    0 & 1 & 0\\
    0 & 0 &-2 
    \end{matrix}
    \right)\Td{x}.
\end{equation}
We use the same parameter as Winkel et al.~\cite[Table 1]{WinkelEtAl2015}.
For~\eqref{eq:motion} with magnetic field~\eqref{eq:B} and electric field~\eqref{eq:E} and a single particle inside the Penning trap, an analytic solution can be found~\cite{BrownGab1986}. 
Note that because of the zero row in the matrix in~\eqref{eq:B}, the force along the third component is independent of the velocity, while the forces along the first or second component are not.
By looking at the error in the first and third component separately, we will confirm below the different convergence orders that our theory predicts for these cases. 

\subsubsection{Local error}
Figure~\ref{fig:local_order_ho} shows the local position (left) and velocity (right) of SDC along the first axis error against the time step $\Delta t$ scaled with the frequency of the magnetic field.
The local error is computed by taking the difference $\Delta x_{i}^{\mathrm{(abs)}}:=|x_{i}^{\mathrm{(approx)}}-x_{i}^{\mathrm{(analyt)}}|$ between numerical and analytic solution after a single time step.
The index $i=1,2, 3$ indicates the two horizontal and one vertical axes.
In line with our theoretical predictions, the order of the local error increases by one for every iteration and the order of the local error in the position is always one higher than the order of the local error in the velocity.
\begin{figure}[t]
	\centering
	\includegraphics[scale=0.34]{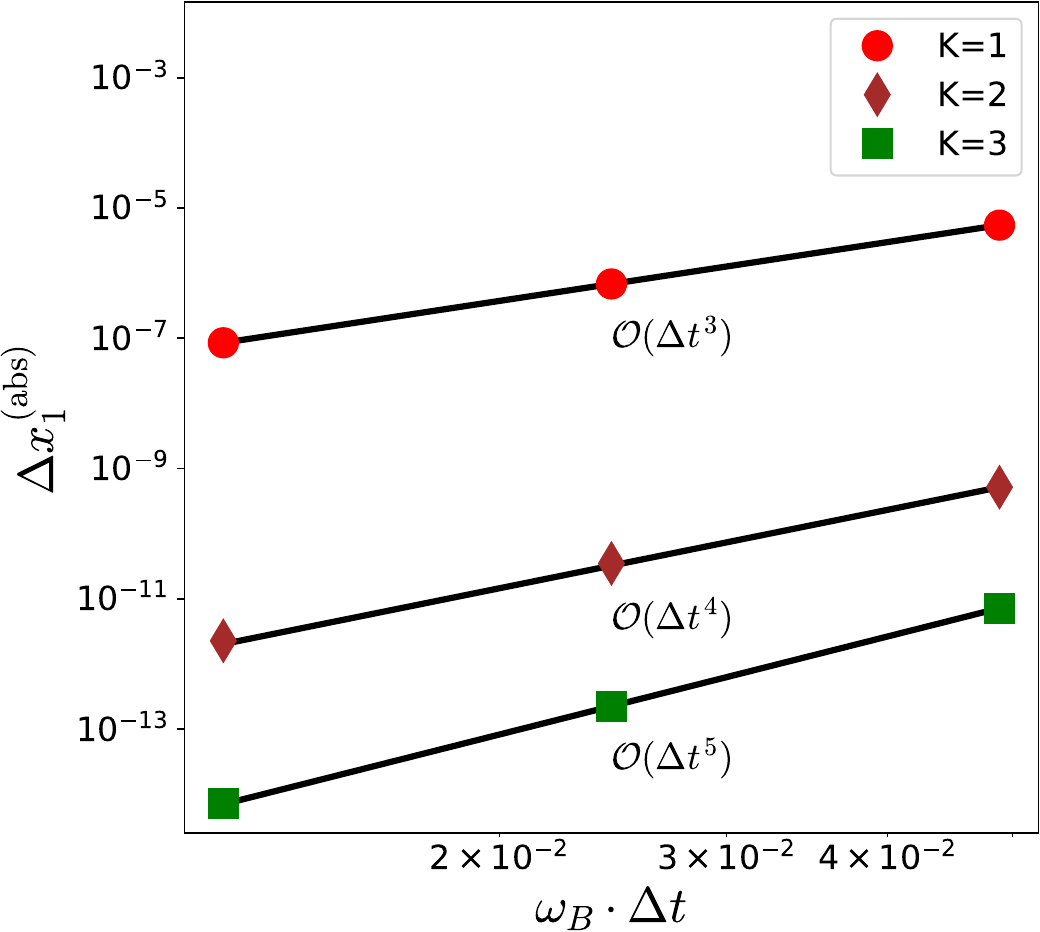}
	\includegraphics[scale=0.34]{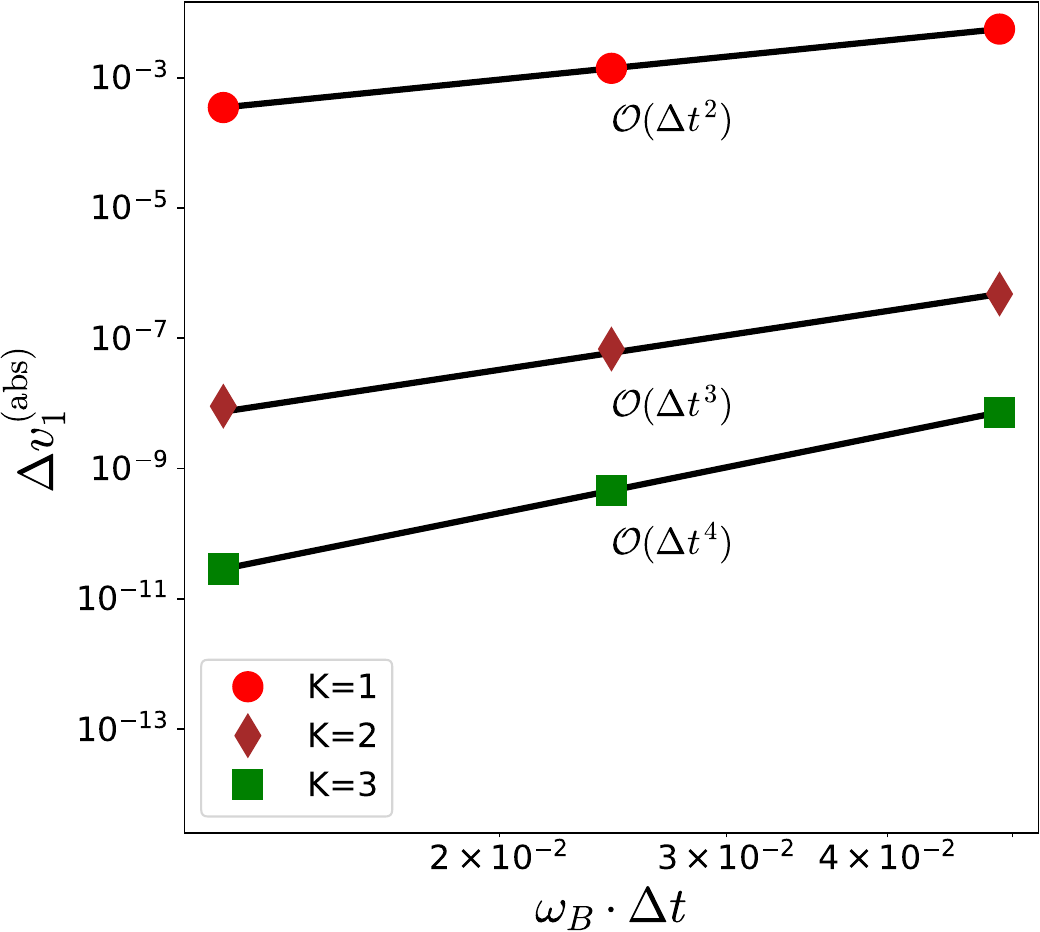}
	\caption{Absolute local error $\Delta x_{1}^{(\mathrm{abs})}$ in the first component of the particle's position (left) and velocity (right) using $K=1,2,3$ SDC iterations and M=5.}
	\label{fig:local_order_ho}
\end{figure}

Figure~\ref{fig:local_order_ver} shows the local error for position and velocity in the third component where the force is independent of $v$ for SDC using $M=5$ Gauss-Legendre quadrature nodes.
As predicted by Theorem~\ref{the:main}, the one order difference between position and velocity error remains, but the order of the local error increases by two orders per iteration. 
\begin{figure}[t]
	\centering
	\includegraphics[scale=0.34]{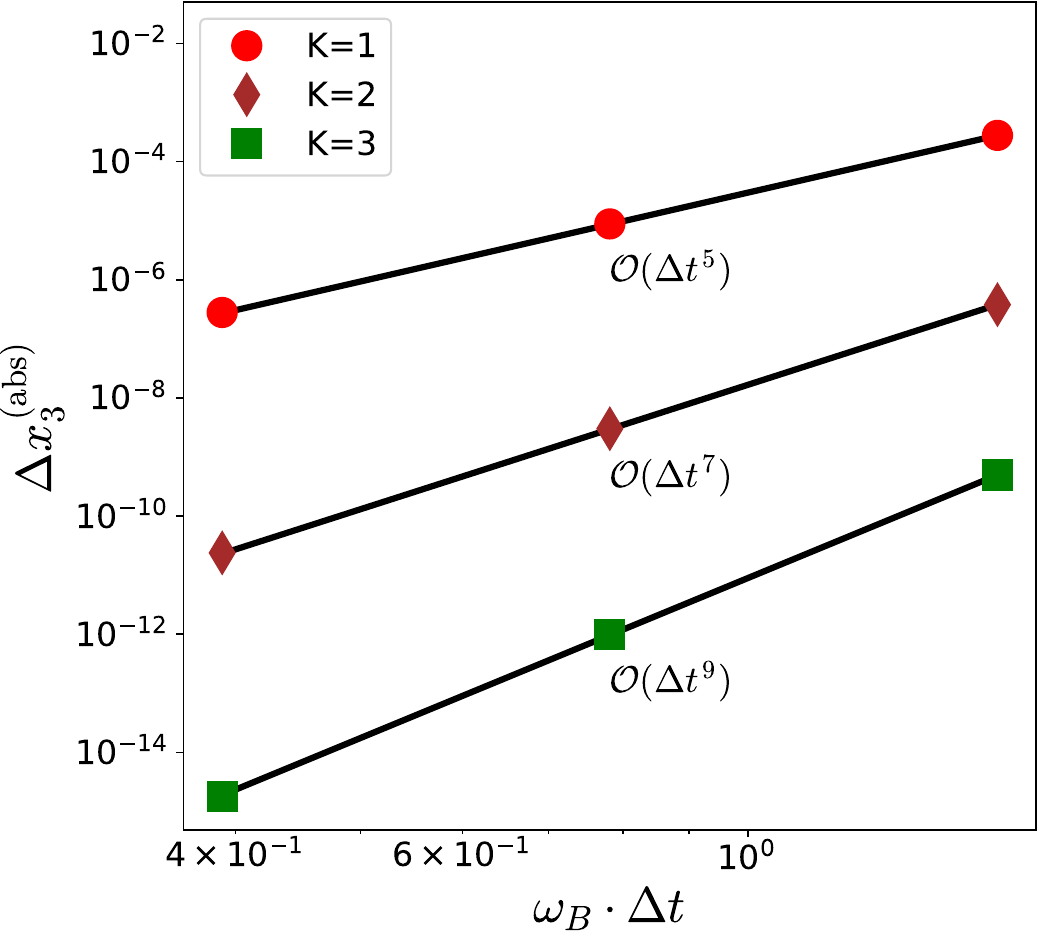}
	\includegraphics[scale=0.34]{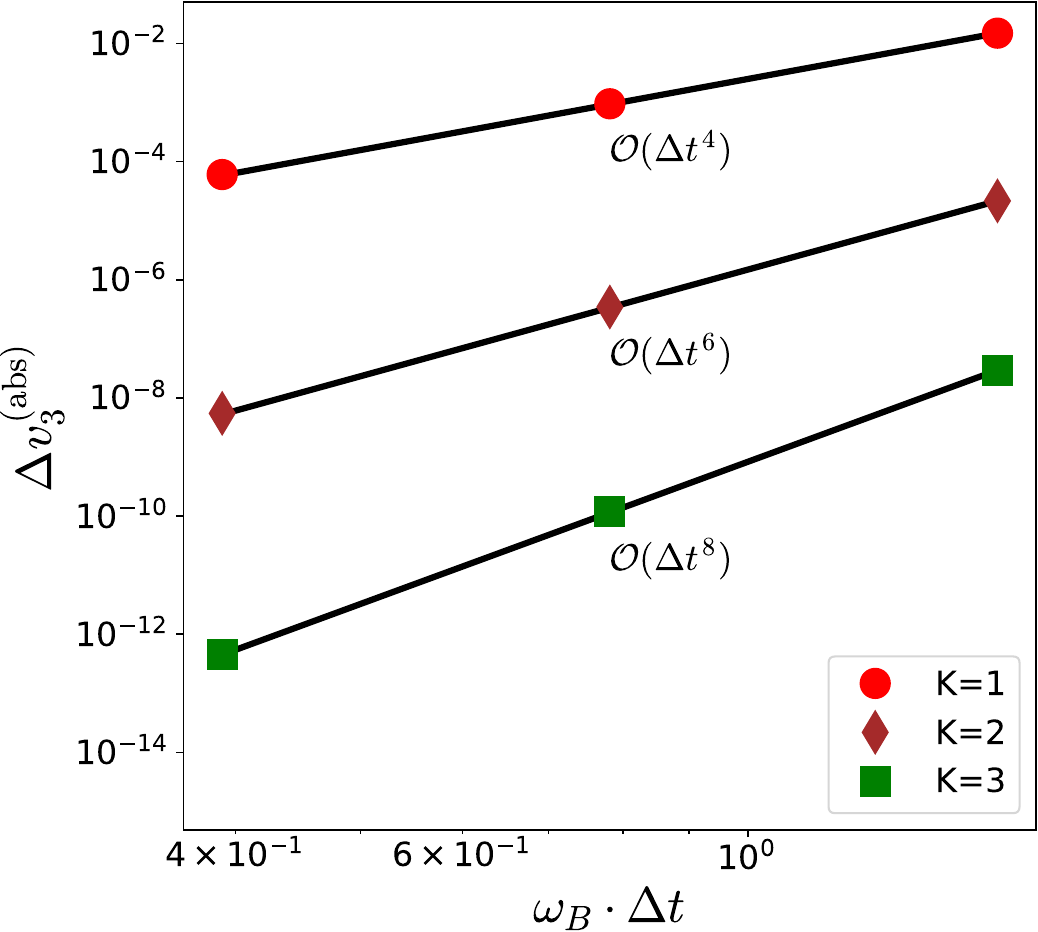}
	\caption{Absolute local error $\Delta x_{3}^{(\mathrm{abs})}$ in the third component of the particle's position (left) and velocity (right) using one, two and three SDC iterations and 5 quadrature nodes. In line with Theorem~\ref{the:main}, the order increases by two per iteration.}
	\label{fig:local_order_ver}
\end{figure}

\subsection{Global error}
Figure~\ref{fig:conv1} shows the relative global error in the $v_{1}$-component (left) and $v_{3}$-component (right) of the velocity for $M=3$ Gauss-Legendre nodes
with fixed final time $t_{\mathrm{end}}=2$.
Since the velocity depends on the position because of the inhomogeneous magnetic- and electric-field, the global error will have the order of the lower local order of the velocity. 
In line with Theorem~\ref{thm:global} we see that in the $v_{1}$-direction every iteration increases the global convergence order by one.
By contrast, in the $v_{3}$-direction, where the force is independent of the velocity, every iteration increases the global order by two.
\begin{figure}[t]
	\centering
	\includegraphics[scale=0.34]{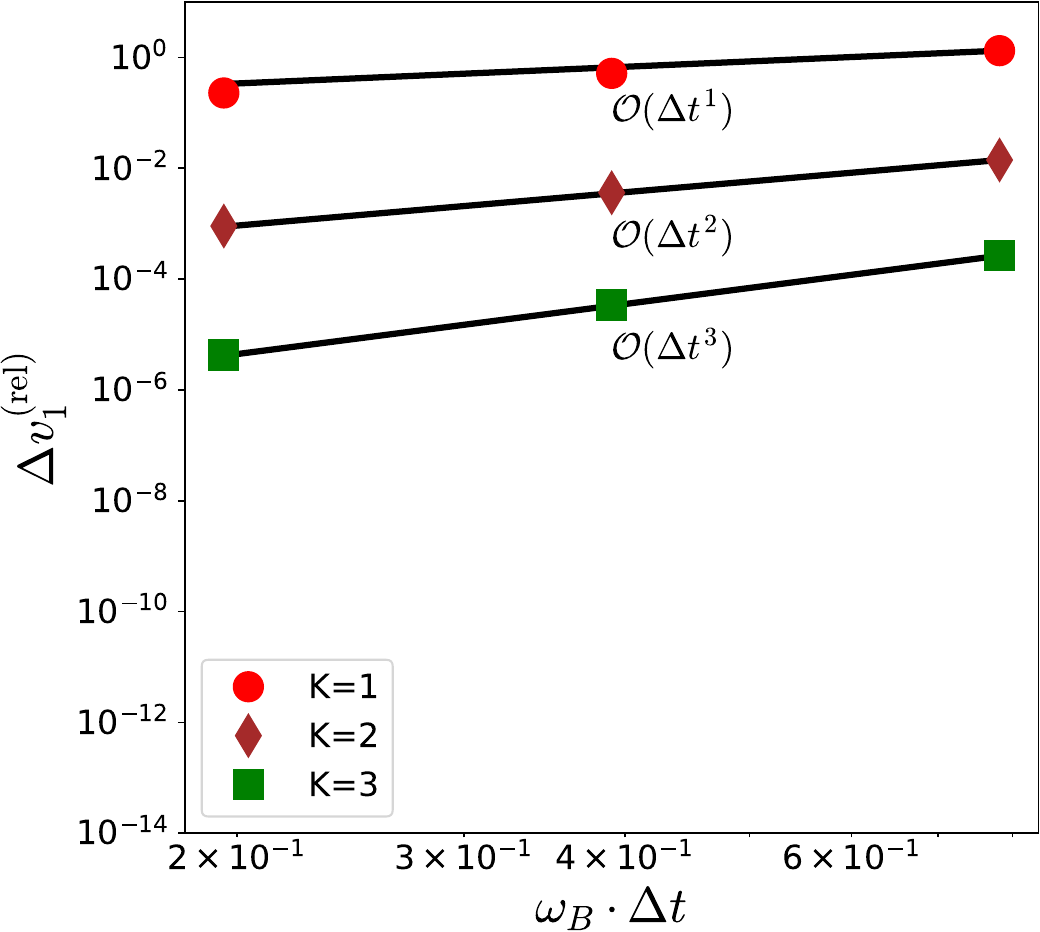}
	\includegraphics[scale=0.34]{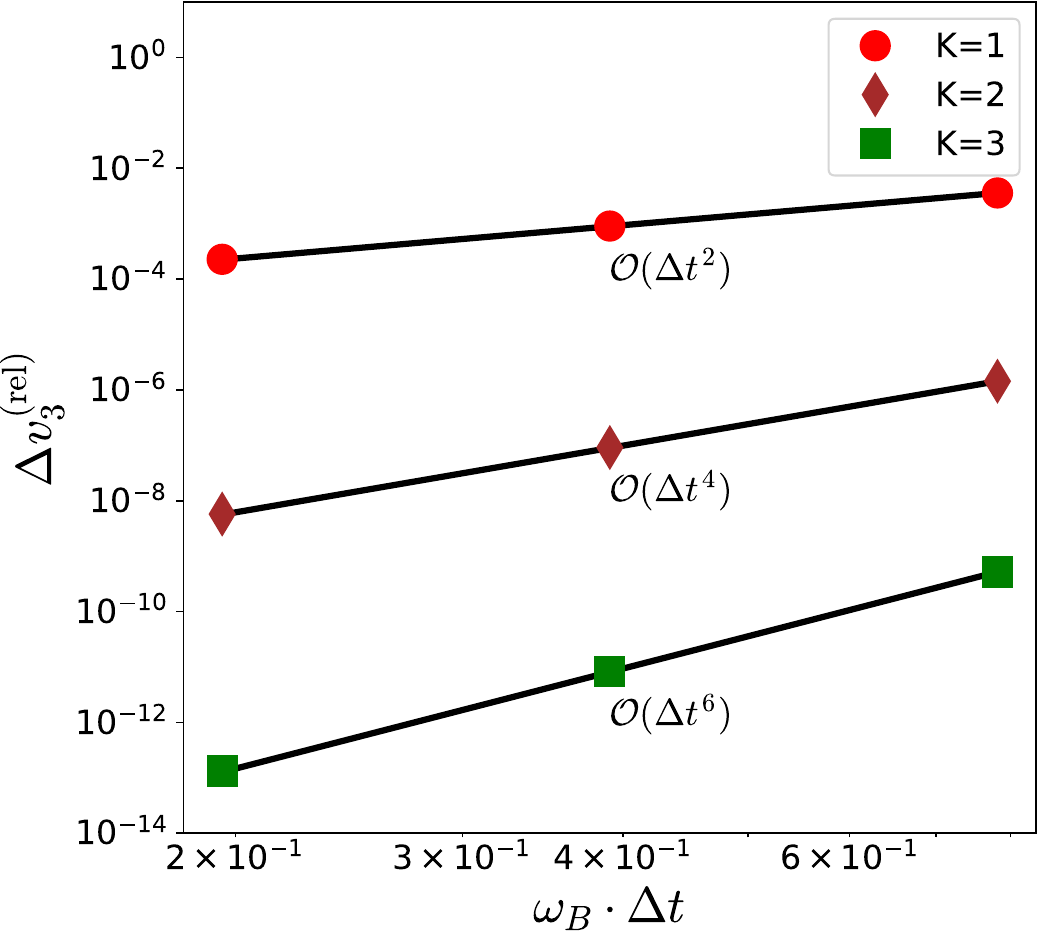}
	\caption{Relative global error $\Delta v_{i}^{(\mathrm{rel})}, \ i=1,3$ in the first component of the particle's horizontal (left) and vertical (right) velocity  in the Penning trap versus time step size for 3 Gauss-Legendre collocation nodes using one, two and three SDC iterations.}
	\label{fig:conv1}
\end{figure}

Table~\ref{table:1} shows measured convergence rates rounded to two digits for $M=2, 3, 4$ nodes and $K=1,2,3$ and $K=10$ iterations.
The theoretically predicted convergence rates according to Theorem~\ref{thm:global} are shown in brackets.
The left table shows the error in the $x_{1}$-component and the right table the error in the $x_{3}$-component.
In line with theory, the order increases by one per iteration in the former and by two per iteration in the latter case.
For $K=10$ iterations, the order is governed by the order of the underlying quadrature rule and therefore the same in both first and third components.
\begin{table}[h]
	\centering
	\caption{Measured convergence rate rounded to two digits followed by convergence rate predicted by Theorem~\ref{the:main} in brackets for different number of Gauss-Legendre quadrature nodes.}
	\label{table:1}
	\begin{tabular}{||c| c |c |c||} 
		\hline
		\multicolumn{4}{|c|}{Horizontal axis} \\
		\hline
		$K$ &  $M=2$ & $M=3$ & $M=4$ \\
		\hline
		$1$ & $1.28 (1)$ & $1.30 (1)$ & $1.61 (1)$ \\
		$2$ & $1.99 (2)$ & $1.99 (2)$ & $2.14 (2)$\\
		$3$ & $2.99 (3)$ & $2.99 (3)$ & $2.98 (3)$\\
		$10$ & $3.99 (4)$ & $5.99 (6)$ & $7.77 (8)$\\
		\hline
	\end{tabular}
	\hspace{5mm}
	\begin{tabular}{||c| c |c |c||} 
		\hline
		\multicolumn{4}{|c|}{Vertical axis} \\
		\hline
		$K$ &  $M=2$ & $M=3$ & $M=4$ \\
		\hline
		$1$ & $1.99 (2)$ & $2.00 (2)$ & $1.99 (2)$ \\
		$2$ & $4.00 (4)$ & $3.99 (4)$ & $3.98 (4)$\\
		$3$ & $3.99 (4)$ & $5.96 (6)$ & $5.97 (6)$\\
		$10$ & $3.99 (4)$ & $5.99 (6)$ & $7.91 (8)$\\
		\hline
	\end{tabular}
	\medskip
	
\end{table}

\subsection{Conservation properties}
\begin{figure}[t]
	\centering
	\includegraphics[scale=0.31]{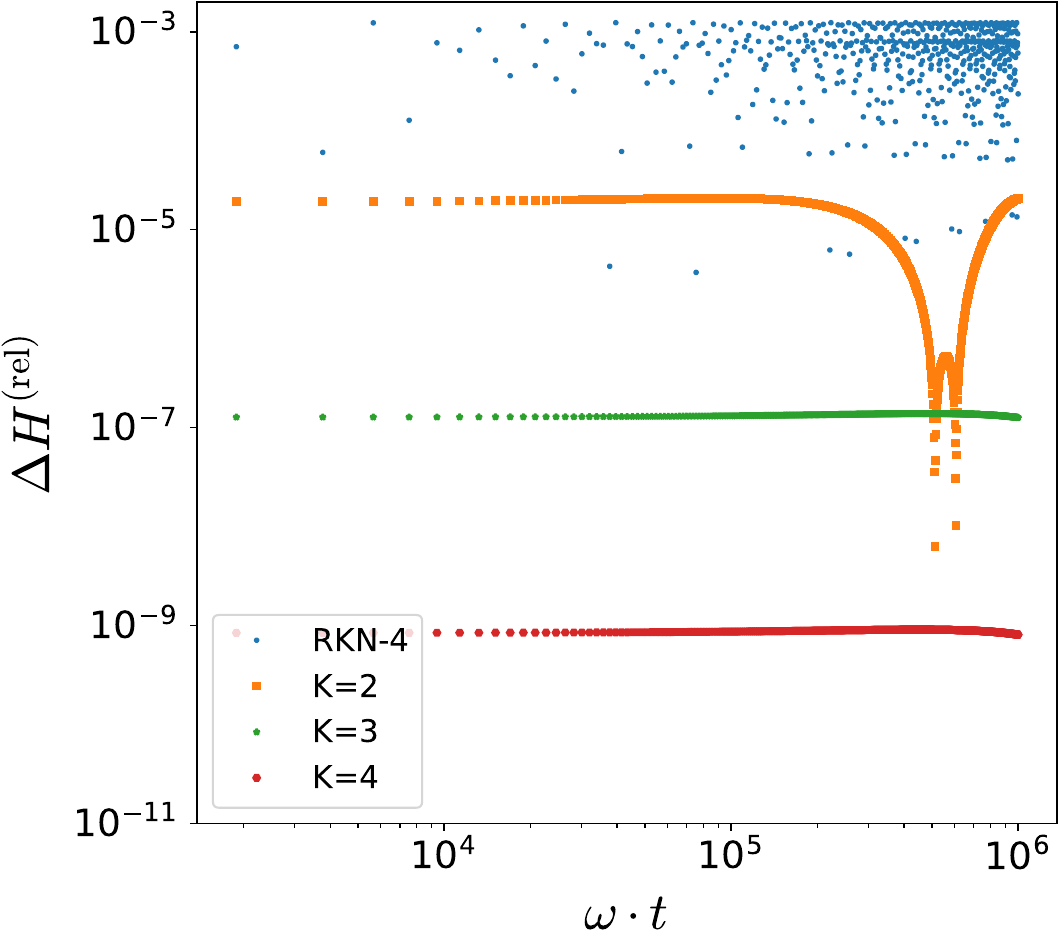}
	\includegraphics[scale=0.31]{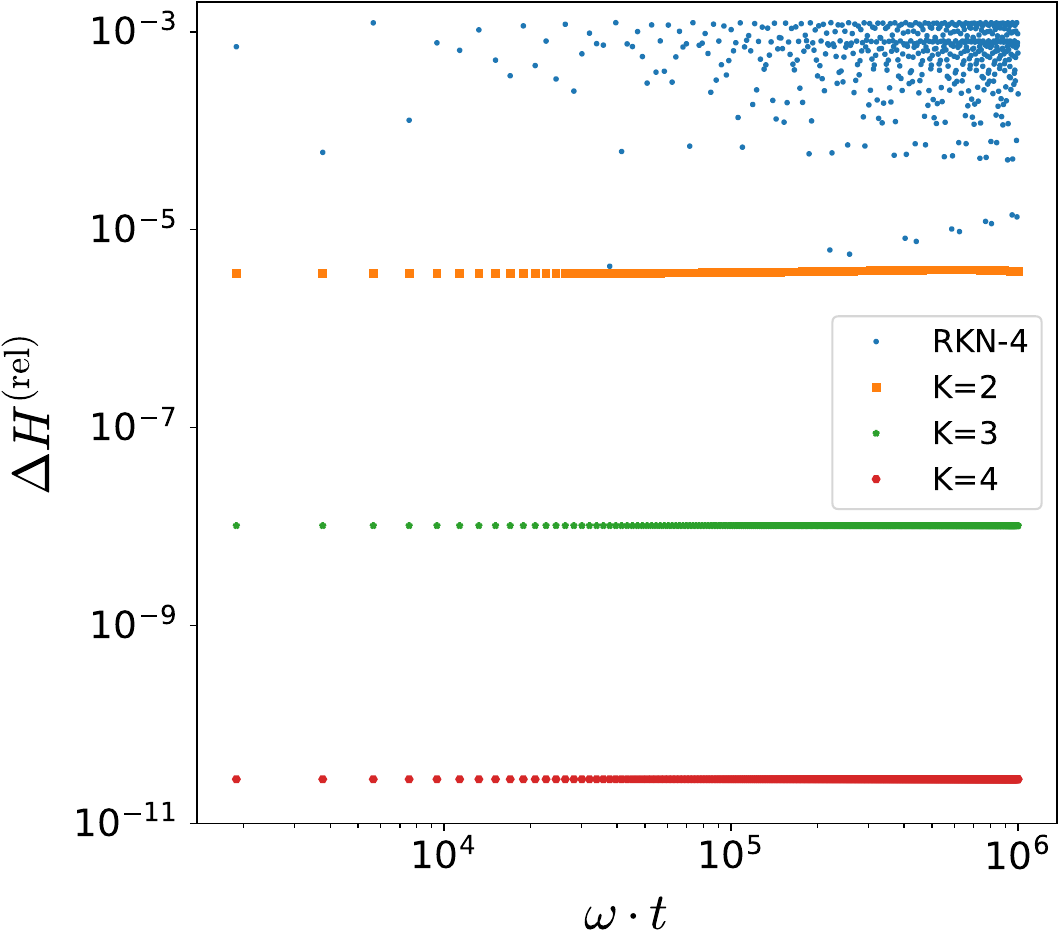}
	\caption{Relative error in the discrete Hamiltonian for the undamped harmonic oscillator over 1.5 million time steps for $M=3$ (left) and $M=5$ (right).}
	\label{fig:hamiltonian}
\end{figure}
Many second-order problems are Hamiltonian systems for which conservation properties of the time integrator are important.
We consider the undamped harmonic oscillator~\eqref{eq:harmonic_oscillator} with $\mu = 0.0$ and $\kappa=1.0$ and a resulting oscillation frequency of $\omega = 1.0$.
The continuous Hamiltonian $H(t) = \frac{1}{2} \left( x(t)^2 + v(t)^2 \right)$ is constant so that $H(t) = H(0)$.
Figure~\ref{fig:hamiltonian} shows the relative error $|H_n - H_0| / H_0$ in the discrete Hamiltonian $H_n = \frac{1}{2} \left( x_n^2 + v_n^2 \right)$ for a time step of $\Delta t = 2 \pi / 10$ until $t_{\text{end}} = 1 \times 10^6$ for a total of 1,591,551 steps for RKN and SDC with $M=3$ and $M=5$ Gauss-Legendre quadrature nodes and $K=2, 3, 4$ iterations.
Since the collocation method is symplectic, we expect bounded long term error for large $K$.
However, already for $K=2$ second-order SDC shows no discernable drift.
Furthermore, the relative error in the Hamiltonian from SDC is smaller than from RKN4 and decreases by about two orders of magnitude per iteration.
This is in line with previous findings for the Lorentz equations that showed low energy errors and little to no drift for SDC, even for very long simulation times~\cite{TretiakRuprecht2019,WinkelEtAl2015}.

\section{Computational efficiency}\label{sec:efficiency}
SDC requires more function evaluation per time step than the Picard iteration or a Runge-Kutta-Nyström (RKN) method.
However, for the same $\Delta t$, it will produce a smaller error.
This allows SDC to achieve accuracy comparable to Picard or RKN-4 with a larger time step.

For a fair comparison in terms of efficiency, Figure~\ref{fig:work_precision} shows the relative error in the first component (left) and third component (right) of the position for the Penning trap benchmark for SDC, Picard iteration and RKN-4 against the total number of $f$ evaluations required.
Note that the different iteration numbers for SDC in the left plot ($K=2,4,6$) and the right plot ($K=1,2,3$) are chosen to achieve the same global convergence rates in both cases.
In all cases, SDC is more efficient than Picard using the same number of iterations.
The advantage of SDC is more pronounced for the case where the force does depend on velocity.
Furthermore, with sufficiently many iterations, the increasing order of SDC allows it to eventually outperform RKN-4.
For the error in the third component, $K=3$ iterations are enough for SDC to become more efficient than RKN-4 while in the first component it requires $K=4$ iterations.
While not shown, the provided code can also generate work-precision results for the velocity-Verlet integrator.
For the error in the first component, we found it to be marginally less efficient than SDC with $k=2$ iterations while for the error in the third component was slightly more efficient than SDC with $k=1$ iteration.
In both cases, the higher-order variants of SDC are significantly more efficient than velocity Verlet (not shown).

\begin{figure}[t]
	\includegraphics[scale=0.31]{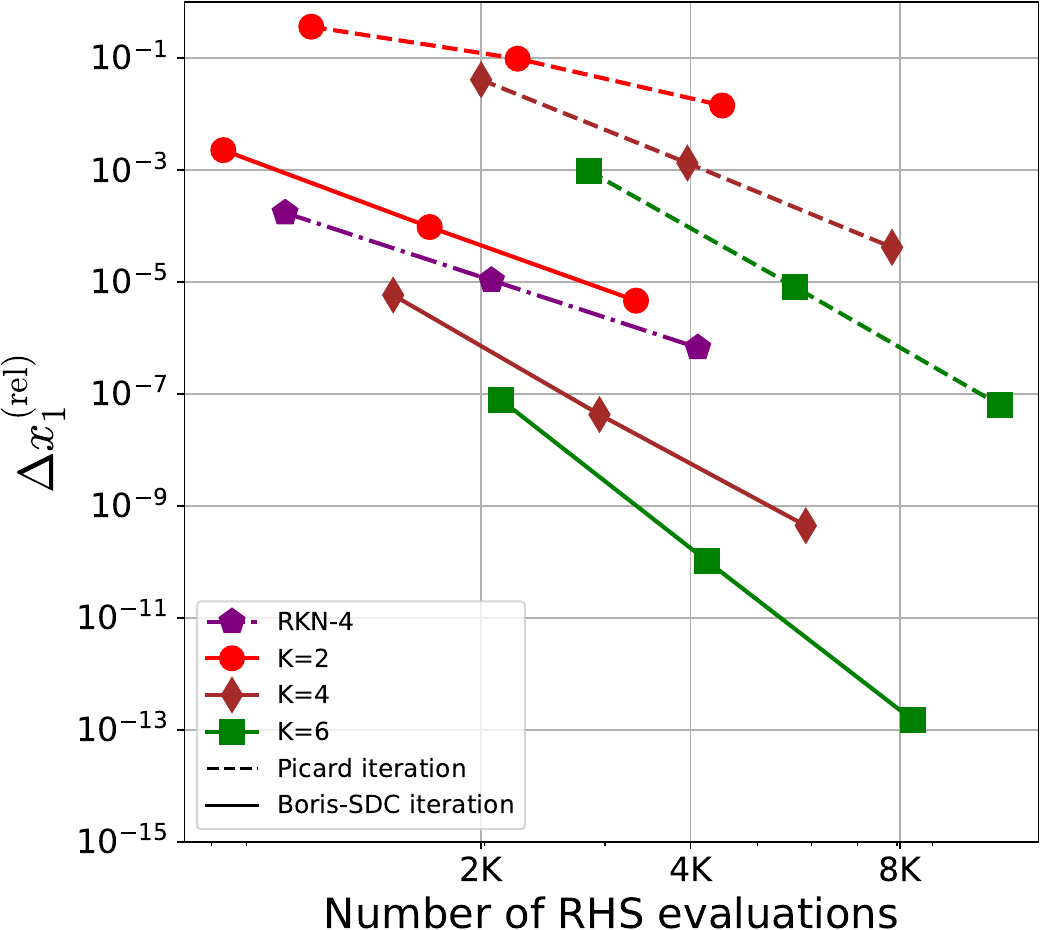}
		\includegraphics[scale=0.31]{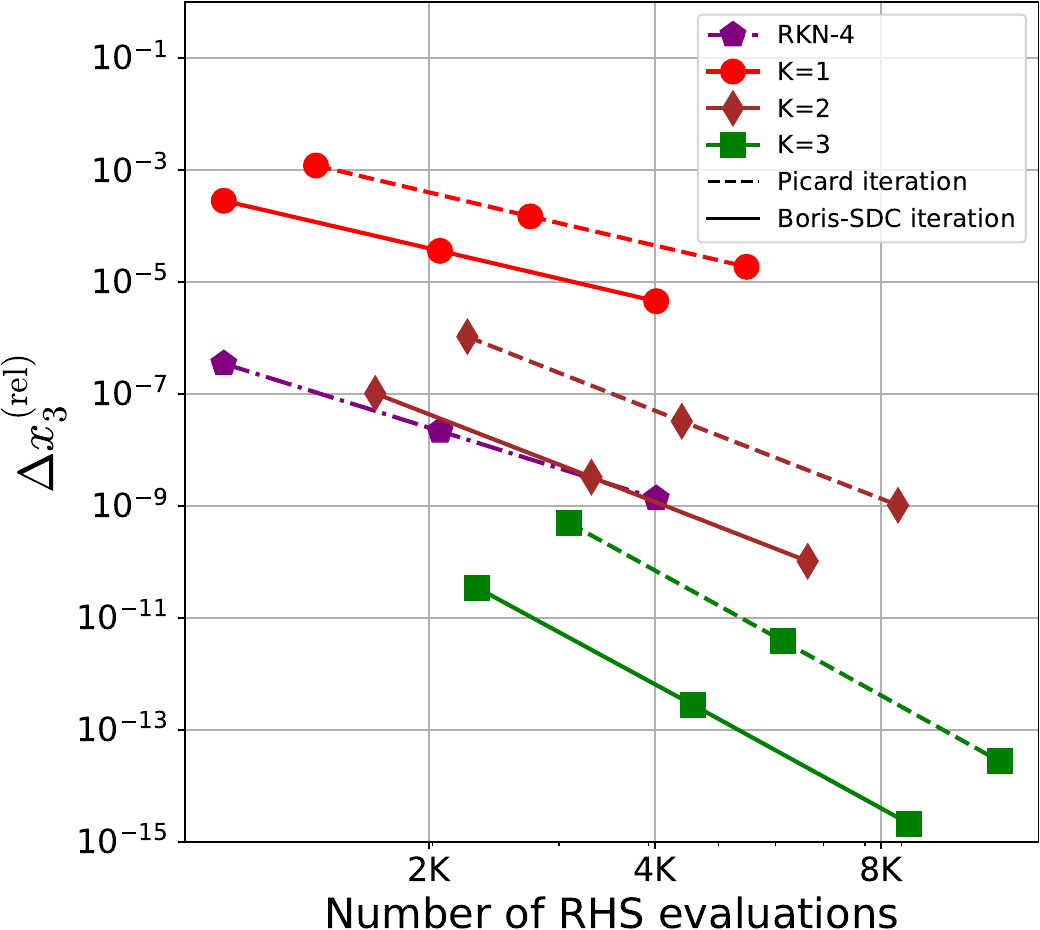}
	\caption{Relative position error in $x_{1}$-direction (left) and $x_{3}$-direction (right) for SDC (solid lines), Picard (dashed lines) with different iteration numbers $K$ with $M=5$ quadrature nodes and RKN-4 against the total number of $f$ evaluations.}
	\label{fig:work_precision}
\end{figure}

\section{Conclusions}
We provide a theoretical analysis of spectral deferred corrections applied to second-order problems.
Using the damped harmonic oscillator as a test problem, similar to how the Dahlquist equation is used for first-order problems, we investigate convergence and stability of SDC compared against a collocation method using Picard iteration and a Runge-Kutta-Nyström-4 method.
The main theoretical result of the paper is a proof that every SDC iteration increases the global convergence order by one for problems where the force depends on the velocity and by two if the force is independent of the velocity.
We also show that the order of the local position error is one higher than the order of the local velocity error.
Our theoretical predictions are validated against numerical examples for the Penning trap benchmark.
We compare SDC against Picard and RKN-4 with respect to work-precision and find it to be competitive for medium to high accuracies.

\section*{Acknowledgments}
We gratefully acknowledge many helpful discussions with Thibaut Lunet as well as his assistance with the pySDC software.
\appendix

\section{Auxiliary results}\label{section:aux_results}
This appendix collects a number of technical results we need for the proof of the main convergence results in Section~\ref{sec:convergence}. 
\begin{proposition}\label{prop:qt-qx}
For the matrices \rev{introduced in~\eqref{eq:dis_coll}},~\eqref{eq:Qt} and~\eqref{eq:Qx}, we have the following bounds
 \begin{subequations}
\begin{align}\label{eq:qt}
&\|Q_{\mathrm{T}}\|\leq 1, \ \quad
	\left\| Q_{\mathrm{x}} \right\| \leq \frac{3}{2}, \\\label{eq:qq}
	&\|Q\|\leq 1,\qquad \|QQ\| \leq 1.
\end{align}
\end{subequations}
\end{proposition}
\begin{proof}
From Ruprecht and Speck~\cite[Lemma 3.1]{RuprechtSpeck2016} we know that
\begin{equation}
    \|Q_{\textrm{I}}\|\leq 1, \ \|Q_{\textrm{E}}\|\leq 1.
\end{equation}
Furthermore, it holds that
\begin{equation*}
	\left\| Q_{\textrm{E}} \circ Q_{\textrm{E}} \right\| \leq \Delta \tau_1^2 + \ldots + \Delta \tau_M^2 \leq \Delta \tau_1 + \ldots + \Delta \tau_M \leq 1.
\end{equation*}
Then,
\begin{equation*}
 \|Q_{\textrm{T}}\|\leq \frac12(\|Q_{\textrm{E}}\|+\|Q_{\textrm{I}}\|)\leq 1
\end{equation*}
and
\begin{equation*}
 \left\| Q_{\textrm{x}} \right\| \leq \left\| Q_{\textrm{E}} \right\| \left\| Q_{\textrm{T}} \right\| + \frac{1}{2} \left\| Q_{\textrm{E}} \right\| \leq \frac{3}{2}.
\end{equation*}
The bounds for the norm of the $Q$ matrix were proven by Caklovic~\cite{Gaya2022}. 
Furthermore, we have $\|QQ\|\leq \|Q\|\|Q\|\leq 1$ which completes the proof.
\end{proof}
\begin{proposition}\label{prop:nodes}
Let $\Td{f}$ be a Lipschitz continuous function with Lipschitz constant $L$ and $(\Tt{X}, \Tt{V})$ be the solution to the collocation problem~\eqref{eq:dis_coll}.
Let $(\Tt{X}^{k}, \Tt{V}^{k})$ be approximations provided by the SDC iteration~\eqref{eq:SDC_vec}. 
Suppose $\Delta t$ satisfies  
\begin{equation}
\Delta t \leq (1 - \delta)/L \mbox{  and  } \Delta t^{2}< \frac13.  	
\end{equation}
for some positive number $0<\delta<1$.
Then, the following holds
\begin{subequations}\label{eq:ineq_nodes}
\begin{align}\label{eq:ineq_nodes_pos}
	&\|\Tt{X}-\Tt{X}^{k}\|\leq C_{1}L\Delta t^{2}(\|\Tt{X}-\Tt{X}^{k-1}\|+\|\Tt{V}-\Tt{V}^{k-1}\|+\|\Tt{V}-\Tt{V}^{k}\|),\\\label{eq:ineq_nodes_vel}
    &\|\Tt{V}-\Tt{V}^{k}\|\leq C_{2}L\Delta t(\|\Tt{V}-\Tt{V}^{k-1}\|+\|\Tt{X}-\Tt{X}^{k-1}\|+\|\Tt{X}-\Tt{X}^{k}\|)
\end{align}
\end{subequations}
with constants $C_1$, $C_2$ independent of $\Delta t$.
If $\Td{f}$ does not depend on $\Td{v}$ we have
\begin{subequations}\label{eq:indep_ineq_nodes}
    \begin{align}\label{eq:indep_ineq_nodes_pos}
	&\|\Tt{X}-\Tt{X}^{k}\|\leq C_{1}L\Delta t^{2}\|\Tt{X}-\Tt{X}^{k-1}\|,\\\label{eq:indep_ineq_nodes_vel}
	&\|\Tt{V}-\Tt{V}^{k}\|\leq 2 L\Delta t(\|\Tt{X}-\Tt{X}^{k-1}\|+\|\Tt{X}-\Tt{X}^{k}\|).
	\end{align}
\end{subequations}
\end{proposition}
\begin{proof}
To prove~\eqref{eq:ineq_nodes_pos}, subtract~\eqref{eq:SDC_vec_pos} from~\eqref{eq:dis_coll_pos} to get
\begin{equation*}
\Tt{X}-\Tt{X}^{k}=\Delta t^2 \Tt{QQ}(F(\Tt{X},\Tt{V})-F(\Tt{X}^{k-1},\Tt{V}^{k-1}))+\Delta t^2 \Tt{Q}_{\textrm{x}}(F(\Tt{X}^{k-1}, \Tt{V}^{k-1})-F(\Tt{X}^{k},\Tt{V}^{k})).
\end{equation*}
Using triangle inequality and Lipschitz continuity we have
\begin{align*}
\|\Tt{X}-\Tt{X}^{k}\| &\leq \|\Tt{QQ}\| L \Delta t^{2}(\|\Tt{X}-\Tt{X}^{k-1}\|+\|\Tt{V}-\Tt{V}^{k-1}\|)+ \\
    &+\|\Tt{Q}_{\textrm{x}}\| L \Delta t^{2}(\|\Tt{X}-\Tt{X}^{k-1}\|+\|\Tt{V}-\Tt{V}^{k-1}\|) \\
    &+\|\Tt{Q}_{\textrm{x}}\| L \Delta t^{2}(\|\Tt{X}-\Tt{X}^{k}\|+\|\Tt{V}-\Tt{V}^{k}\|).
\end{align*}
Since $\|\Tt{Q}_{\textrm{x}}\|\leq \frac32$ and $\|\Tt{QQ}\|\leq 1$ by Proposition~\ref{prop:qt-qx}, we can simplify to
\begin{align*}\label{eq:X}
    \|\Tt{X}-\Tt{X}^{k}\| &\leq (1+\frac32) L \Delta t^{2}(\|\Tt{X}-\Tt{X}^{k-1}\|+\|\Tt{V}-\Tt{V}^{k-1}\|)\\
    &+ \quad \frac32L\Delta  t^{2}(\|\Tt{X}-\Tt{X}^{k}\|+\|\Tt{V}-\Tt{V}^{k}\|).
\end{align*}
Because of $\Delta t^{2}<\frac13$ we have $1-\frac32 L\Delta t^{2}\geq 1-L\Delta t\geq \delta >0$ and thus
\begin{equation*}
    \|\Tt{X}-\Tt{X}^{k}\|\leq \frac{5L\Delta t^{2}}{2-3L\Delta t^{2}}(\|\Tt{X}-\Tt{X}^{k-1}\|+\|\Tt{V}-\Tt{V}^{k-1}\|)+\frac{3L\Delta t^{2}}{2-3L\Delta t^{2}}\|\Tt{V}-\Tt{V}^{k}\|.
\end{equation*}
Since $1-\frac32 L\Delta t^{2}>\delta$ for $\Delta t^{2}<\frac13$, this yields
\begin{equation}
\frac{5}{2-3\Delta t^{2}L}\leq \frac{5}{2\delta}=:C_{1}.
\end{equation}
Hence,
\begin{equation}
    \|\Tt{X}-\Tt{X}^{k}\|\leq C_{1} L \Delta t^2 (\|\Tt{X}-\Tt{X}^{k-1}\|+\|\Tt{V}-\Tt{V}^{k-1}\|+\|\Tt{V}-\Tt{V}^{k}\|).
\end{equation}
We can prove~\eqref{eq:ineq_nodes_vel} in a similar way. 
Subtract~\eqref{eq:dis_coll_vel} from~\eqref{eq:SDC_vec_vel} to get
\begin{equation*}
\Tt{V}-\Tt{V}^{k}=\Delta t \Tt{Q}(F(\Tt{X},\Tt{V})-F(\Tt{X}^{k-1},\Tt{V}^{k-1}))+\Delta t \Tt{Q}_{\textrm{T}}(F(\Tt{X}^{k-1},\Tt{V}^{k-1})-F(\Tt{X}^{k},\Tt{V}^{k})).
\end{equation*}
Then,
\begin{align*}
    \|\Tt{V}-\Tt{V}^{k}\| &\leq \Delta t \|\Tt{Q}\|\|F(\Tt{X},\Tt{V})-F(\Tt{X}^{k-1},\Tt{V}^{k-1})\|\\
    &\quad +\Delta t \|\Tt{Q}_{\textrm{T}}\|\|F(\Tt{X}^{k-1},\Tt{V}^{k-1})-F(\Tt{X}^{k},\Tt{V}^{k})\|.
\end{align*}
By using that $F$ is Lipschitz continuous we obtain
\begin{multline*}
     \|\Tt{V}-\Tt{V}^{k}\|\leq \|\Tt{Q}\| L \Delta t (\|\Tt{X}-\Tt{X}^{k-1}\|+\|\Tt{V}-\Tt{V}^{k-1}\|)\\+\|\Tt{Q}_{\textrm{T}}\|L \Delta t (\|\Tt{X}-\Tt{X}^{k-1}\|+\|\Tt{V}-\Tt{V}^{k-1}\|+\|\Tt{X}-\Tt{X}^{k}\|+\|\Tt{V}-\Tt{V}^{k}\|).
\end{multline*}
Since $\|\Tt{Q}_{\textrm{T}}\|\leq 1$ and $\|\Tt{Q}\|\leq 1$ by Proposition~\ref{prop:qt-qx}, it follows that
\begin{multline*}\label{eq:v1}
     \|\Tt{V}-\Tt{V}^{k}\|\leq 2L\Delta t (\|\Tt{X}-\Tt{X}^{k-1}\|+\|\Tt{V}-\Tt{V}^{k-1}\|)+L\Delta t(\|\Tt{X}-\Tt{X}^{k}\|+\|\Tt{V}-\Tt{V}^{k}\|).
\end{multline*}
Since, $1-L\Delta t \geq\delta>0$, we get
\begin{multline*}
     \|\Tt{V}-\Tt{V}^{k}\|\leq \frac{2L \Delta t }{1-L \Delta t }(\|\Tt{X}-\Tt{X}^{k-1}\|+\|\Tt{V}-\Tt{V}^{k-1}\|)+\frac{L\Delta t }{1- L \Delta t}(\|\Tt{X}-\Tt{X}^{k}\|).
\end{multline*}
Because $1-L\Delta t \geq\delta>0$, we have
\begin{equation}
\frac{2}{1- L \Delta t }\leq \frac{2}{\delta}=:C_{2}
\end{equation}
and
\begin{equation}\label{eq:vel_ineq}
    \|\Tt{V}-\Tt{V}^{k}\|\leq C_{2} L \Delta t (\|\Tt{X}-\Tt{X}^{k-1}\|+\|\Tt{V}-\Tt{V}^{k-1}\|+\|\Tt{X}-\Tt{X}^{k}\|).
\end{equation}
Let $\Td{f}$ be independent of $\Td{v}$, i.~e., $F(\Tt{X},\Tt{V})=F(\Tt{X})$. Subtracting~\eqref{eq:dis_coll} from~\eqref{eq:SDC_vec} yields
\begin{subequations}
	\begin{align}
	\Tt{X}-\Tt{X}^{k}=\Delta t^2 \Tt{QQ}(F(\Tt{X})-F(\Tt{X}^{k-1}))+\Delta t^2 \Tt{Q}_{\textrm{x}}(F(\Tt{X}^{k-1})-F(\Tt{X}^{k})),\\
	\Tt{V}-\Tt{V}^{k}=\Delta t \Tt{Q}(F(\Tt{X})-F(\Tt{X}^{k-1}))+\Delta t \Tt{Q}_{\textrm{T}}(F(\Tt{X}^{k-1})-F(\Tt{X}^{k})).
	\end{align}
\end{subequations}
Using similar arguments as above, we obtain
\begin{subequations}
\begin{align}
    \|\Tt{X}-\Tt{X}^{k}\|\leq C_{1} L\Delta t^2 \|\Tt{X}-\Tt{X}^{k-1}\|\\
    \|\Tt{V}-\Tt{V}^{k}\|\leq 2 L\Delta t  (\|\Tt{X}-\Tt{X}^{k-1}\|+\|\Tt{X}-\Tt{X}^{k}\|)
\end{align}
\end{subequations}
\end{proof}
The following theorem provides the error bound for SDC at the quadrature nodes.
\begin{theorem}\label{Th:bound}
Consider the initial value problem~\eqref{eq:ode_system} and let $\Td{f}$ be Lipschitz continuous with Lipschitz constant $L$.
If the step size $\Delta t$ is sufficiently small, we have
\begin{subequations}
\begin{align}
    &\|\Tt{X}-\Tt{X}^{k}\|\leq \tilde{C_{1}}L^{k} \Delta t^{k+k_0+1},\\
    &\|\Tt{V}-\Tt{V}^{k}\|\leq \tilde{C_{2}}L^{k} \Delta t^{k+k_0}.
\end{align}	
\end{subequations}
with constants $\tilde{C_{1}},\ \tilde{C_{2}}$ independent of $\Delta t$, and $k_{0}$ the order of the procedure used to generate the starting value $\Tt{U}^0$ for the SDC iteration, see Remark~\ref{remark:starting_values}.
If $f$ is independent of $v$, we have 
\begin{subequations}
\begin{align}
	  &\|\Tt{X}-\Tt{X}^{k}\| \leq \hat{C_{1}}L^{k}\Delta t^{2k+k_0},\\
    	&\|\Tt{V}-\Tt{V}^{k}\|\leq \hat{C}_{2}L^{k}\Delta t^{2k+k_{0}-1}.
\end{align}	
\end{subequations}
with constants $\hat{C_{1}},\ \hat{C_{2}}$ independent of $\Delta t$.
\end{theorem} 
\begin{proof}
First, consider a case where the right-hand side function $f$ does not depend on $v$.
Insert~\eqref{eq:indep_ineq_nodes_pos} into~\eqref{eq:indep_ineq_nodes_vel} to get
\begin{align*}
	\|\Tt{V}-\Tt{V}^{k}\|\leq 2L\Delta t(\|\Tt{X}-\Tt{X}^{k-1}\|+C_{1}L\Delta t^{2}\|\Tt{X}-\Tt{X}^{k-1}\|)= \\ =2L\Delta t\|\Tt{X}-\Tt{X}^{k-1}\|+2C_{1}L^{2}\Delta t^{3} \|\Tt{X}-\Tt{X}^{k-1}\|.
\end{align*}
As before,
\begin{equation*}
	\|\Tt{X}-\Tt{X}^{k}\|\leq C_{1}L\Delta t^{2}\|\Tt{X}-\Tt{X}^{k-1}\|,
\end{equation*}
such that
\begin{equation*}
	\|\Tt{V}-\Tt{V}^{k}\|\leq 2L\Delta t \|\Tt{X}-\Tt{X}^{k-1}\|+2C_{1}L^{2}\Delta t^{3} \|\Tt{X}-\Tt{X}^{k-1}\|.
\end{equation*}
By recursive insertion, we get
\begin{align}\label{eq:pos_X0}
	\|\Tt{X}-\Tt{X}^{k}\| &\leq C_{1}^{k}L^{k}\Delta t^{2k}\|\Tt{X}-\Tt{X}^{0}\|,\\\label{eq:vel_V0}
	\|\Tt{V}-\Tt{V}^{k}\| &\leq 2C_{1}^{k-1}L^{k} \Delta t^{2k-1} \|\Tt{X}-\Tt{X}^{0}\|+2C_{1}^{k} L^{k+1}\Delta t^{2k+1} \|\Tt{X}-\Tt{X}^{0}\|. 
\end{align}
For a starting value $\Tt{X}^{0}$ for the SDC iteration of order $k_0$ we have
\begin{subequations}\label{eq:k_0}
\begin{align}
	&\|\Tt{X}-\Tt{X}^{0}\|\leq C_{0}\Delta t^{k_{0}},\\
	&\|\Tt{V}-\Tt{V}^{0}\|\leq C_{0}\Delta t^{k_{0}}
\end{align}
\end{subequations}
where the constant $C_{0}$ is independent of $\Delta t$.
Taken together, we find that
\begin{align*}
\|\Tt{X}-\Tt{X}^{k}\| \leq C_{1}^{k}L^{k}\Delta t^{2k}\|\Tt{X}-\Tt{X}^{0}\| \leq C_{1}^{k}C_{0}L^{k}\Delta t^{2k+k_0}.
\end{align*}
Similarly, using~\eqref{eq:k_0} in equation \eqref{eq:vel_V0} we obtain
\begin{align*}
\|\Tt{V}-\Tt{V}^{k}\| &\leq 2C_{1}^{k-1}L^{k} \Delta t^{2k-1} \|\Tt{X}-\Tt{X}^{0}\|+2C_{1}^{k} L^{k+1}\Delta t^{2k+1} \|\Tt{X}-\Tt{X}^{0}\|  \\
	&\leq 2C_{1}^{k-1}C_{0}L^{k}\Delta t^{2k+k_{0}-1}+ 2C_{1}^{k}C_{0}L^{k+1}\Delta t^{2k+k_{0}+1} \\
	&=2C_{0}C_{1}^{k-1}(1+C_{1}L\Delta t ^{2})L^{k}\Delta t^{2k+k_{0}-1}.
\end{align*}
Since $\Delta t^{2}<\frac13$, the following estimate is valid
\begin{equation*}
2C_{0}C_{1}^{k-1}(1+C_{1}L\Delta t ^{2})\leq2C_{0}C_{1}^{k-1}(1+\frac{C_{1}L}{3})=:\hat{C_{2}}.
\end{equation*}
Thus,
\begin{align*}
	    &\|\Tt{X}-\Tt{X}^{k}\| \leq \hat{C_{1}}L^{k}\Delta t^{2k+k_0},\\
	    &\|\Tt{V}-\Tt{V}^{k}\|\leq \hat{C}_{2}L^{k}\Delta t^{2k+k_{0}-1}
\end{align*}
where $\hat{C}_{1}:=C_{1}^{k}C_{0}$.

For the general case where $\Td{f}$ depends on $\Td{v}$, we use estimate~\eqref{eq:ineq_nodes} in Proposition~\ref{prop:nodes}.
First, we insert $\|\Tt{X}-\Tt{X}^{k}\|$ from~\eqref{eq:ineq_nodes_pos} on the right--hand side of the inequality~\eqref{eq:ineq_nodes_vel} to get
\begin{align*}
	    \|\Tt{X}-\Tt{X}^{k}\| &\leq C_{1}L\Delta t^{2}(\|\Tt{X}-\Tt{X}^{k-1}\|+\|\Tt{V}-\Tt{V}^{k-1}\| \\
	    & \quad +C_{2}L\Delta t(\|\Tt{V}-\Tt{V}^{k-1}\|+\|\Tt{X}-\Tt{X}^{k-1}\|+\|\Tt{X}-\Tt{X}^{k}\|)) \\
	    &=(C_{1}L\Delta t^{2}+C_{1}C_{2}L^{2}\Delta t^{3})(\|\Tt{X}-\Tt{X}^{k-1}\|+\|\Tt{V}-\Tt{V}^{k-1}\|)\\
	    &\quad +C_{1}C_{2}L^{2}\Delta t^{3}\|\Tt{X}-\Tt{X}^{k}\|.
\end{align*}
If $\Delta t$ is small enough such that $1-C_{1} C_{2}L^{2}\Delta t^{3}\geq \tilde{\delta} >0$, then
\begin{equation}\label{eq:2.36}
	    \|\Tt{X}-\Tt{X}^{k}\|\leq \frac{C_{1}L\Delta t^{2}+C_{1}C_{2}L^{2}\Delta t^{3}}{1-C_{1}C_{2}L^{2}\Delta t^{3}}(\|\Tt{X}-\Tt{X}^{k-1}\|+\|\Tt{V}-\Tt{V}^{k-1}\|).
\end{equation}
Analogously, substitute the expression for $\|\Tt{V}-\Tt{V}^{k}\|$ from~\eqref{eq:ineq_nodes_vel} into~\eqref{eq:ineq_nodes_pos} to get
\begin{align*}
	    \|\Tt{V}-\Tt{V}^{k}\| &\leq C_{2}L\Delta t(\|\Tt{V}-\Tt{V}^{k-1}\|+\|\Tt{X}-\Tt{X}^{k-1}\| \\
	    &\quad +C_{1}L\Delta t^{2}(\|\Tt{X}-\Tt{X}^{k-1}\|+\|\Tt{V}-\Tt{V}^{k-1}\|+\|\Tt{V}-\Tt{V}^{k}\|)) \\
	    &=(C_{2}L\Delta t+C_{1}C_{2}L^{2}\Delta t^{3})(\|\Tt{X}-\Tt{X}^{k-1}\|+\|\Tt{V}-\Tt{V}^{k-1}\|)\\
	    &\quad +C_{1}C_{2}L^{2}\Delta t^{3}\|\Tt{V}-\Tt{V}^{k}\|.
\end{align*}
Hence,
\begin{equation}
	    \|\Tt{V}-\Tt{V}^{k}\|\leq \frac{C_{2}L\Delta t+C_{1}C_{2}L^{2}\Delta t^{3}}{1-C_{1}C_{2}L^{2}\Delta t^{3}}(\|\Tt{X}-\Tt{X}^{k-1}\|+\|\Tt{V}-\Tt{V}^{k-1}\|).
\end{equation}
Let
\begin{equation}
	\ m_{1}:=\frac{L\Delta t^{2}(C_{1}+C_{1}C_{2}L\Delta t)}{1-C_{1}C_{2}L^{2}\Delta t^{3}}, \ m_{2}:=\frac{L\Delta t(C_{2}+C_{1}C_{2}L\Delta t^{2})}{1-C_{1}C_{2}L^{2}\Delta t^{3}}.
\end{equation}
and we obtain the following system of inequalities
\begin{align*}
	    \|\Tt{X}-\Tt{X}^{k}\|\leq m_{1}(\|\Tt{X}-\Tt{X}^{k-1}\|+\|\Tt{V}-\Tt{V}^{k-1}\|),\\
	    \|\Tt{V}-\Tt{V}^{k}\|\leq m_{2}(\|\Tt{X}-\Tt{X}^{k-1}\|+\|\Tt{V}-\Tt{V}^{k-1}\|).
\end{align*}
These can be written in matrix form 
\begin{equation}
	    \begin{pmatrix}
	    \|\Tt{X}-\Tt{X}^{k}\|\\
	    \|\Tt{V}-\Tt{V}^{k}\|
	    \end{pmatrix} \leq \begin{pmatrix}
	    m_{1} & m_{1}\\
	    m_{2} & m_{2}
	    \end{pmatrix} \begin{pmatrix}
	    \|\Tt{X}-\Tt{X}^{k-1}\|\\
	    \|\Tt{V}-\Tt{V}^{k-1}\|
	    \end{pmatrix}.
\end{equation}
Recursive insertion yields 
\begin{equation}\label{eq:M_matrix}
	    \begin{pmatrix}
	    \|\Tt{X}-\Tt{X}^{k}\|\\
	    \|\Tt{V}-\Tt{V}^{k}\| 
	    \end{pmatrix}  \leq \begin{pmatrix}
	    m_{1} & m_{1}\\
	    m_{2} & m_{2}
	    \end{pmatrix}^{k} \begin{pmatrix}
	    \|\Tt{X}-\Tt{X}^{0}\|\\
	    \|\Tt{V}-\Tt{V}^{0}\| 
	    \end{pmatrix} =: M^k \begin{pmatrix}
	    \|\Tt{X}-\Tt{X}^{0}\|\\
	    \|\Tt{V}-\Tt{V}^{0}\| 
	    \end{pmatrix}
\end{equation}
It is easy to show by induction that $M^{k}=(m_{1}+m_{2})^{k-1}M$ so that~\eqref{eq:M_matrix} becomes
\begin{equation}
	    \left(\begin{matrix}
	    \|\Tt{X}-\Tt{X}^{k}\|\\
	    \|\Tt{V}-\Tt{V}^{k}\|
	    \end{matrix} \right)\leq \left(\begin{matrix}
	    (m_{1}+m_{2})^{k-1} m_{1} &(m_{1}+m_{2})^{k-1}  m_{1}\\
	    (m_{1}+m_{2})^{k-1} m_{2} &(m_{1}+m_{2})^{k-1}  m_{2}
	    \end{matrix} \right)\left(\begin{matrix}
	    \|\Tt{X}-\Tt{X}^{0}\|\\
	    \|\Tt{V}-\Tt{V}^{0}\|
	    \end{matrix} \right)
\end{equation}
or
\begin{align*}
	    \|\Tt{X}-\Tt{X}^{k}\|\leq m_{1}(m_{1}+m_{2})^{k-1}(\|\Tt{X}-\Tt{X}^{0}\|+\|\Tt{V}-\Tt{V}^{0}\|),\\
	    \|\Tt{V}-\Tt{V}^{k}\|\leq m_{2}(m_{1}+m_{2})^{k-1}(\|\Tt{X}-\Tt{X}^{0}\|+\|\Tt{V}-\Tt{V}^{0}\|).
\end{align*}
For $\Delta t<\frac{1}{\sqrt{3}}$ and $1-C_{1} C_{2}L^{2}\Delta t^{3}\geq \tilde{\delta} >0$, we can write
\begin{align*}
	\frac{C_{2}+C_{1}C_{2}L\Delta t^{2}}{1-C_{1}C_{2}L^{2}\Delta t^{3}}\leq \frac{3C_{2}+C_{1}C_{2}L}{3\tilde{\delta}}=:C_{3},\\
	\frac{C_{1}+C_{1}C_{2}L\Delta t}{1-C_{1}C_{2}L^{2}\Delta t^{3}}\leq \frac{\sqrt{3}C_{1}+C_{1}C_{2}L}{\sqrt{3}\tilde{\delta}}=:C_{4},
\end{align*}
and get the following inequalities
\begin{align*}
	    &\|\Tt{X}-\Tt{X}^{k}\|\leq C_{4}L\Delta t^{2}(C_{3}L\Delta t+C_{4}L\Delta t^{2})^{k-1}(\|\Tt{X}-\Tt{X}^{0}\|+\|\Tt{V}-\Tt{V}^{0}\|),\\
	    &\|\Tt{V}-\Tt{V}^{k}\|\leq C_{3}L\Delta t(C_{3}L\Delta t+C_{4}L\Delta t^{2})^{k-1}(\|\Tt{X}-\Tt{X}^{0}\|+\|\Tt{V}-\Tt{V}^{0}\|).
\end{align*}
Using the results from~\eqref{eq:k_0} yields
\begin{align*}
	    \|\Tt{X}-\Tt{X}^{k}\| &\leq L^{k}\Delta t^{k+1}C_{4}(C_{3}+C_{4}\Delta t)^{k-1}(\|\Tt{X}-\Tt{X}^{0}\|+\|\Tt{V}-\Tt{V}^{0}\|) \\
	    &\leq L^{k}\Delta t^{k+1}C_{4}(C_{3}+C_{4}\Delta t)^{k-1}(C_{0}\Delta t^{k_0}+C_{0}\Delta t^{k_0}) \\
	    &\leq 2C_{4}C_{0}(C_{3}+C_{4}\Delta t)^{k-1}L^{k}\Delta t^{k+k_{0}+1}.
\end{align*}
Similar computations can be done for the variable $\Tt{V}$ which yields
\begin{align*}
	    \|\Tt{V}-\Tt{V}^{k}\| &\leq  L^{k}\Delta t^{k}C_{3}(C_{3}+C_{4}\Delta t)^{k-1}(\|\Tt{X}-\Tt{X}^{0}\|+\|\Tt{V}-\Tt{V}^{0}\|)\\
	    &\leq 2C_{3}C_{0}(C_{3}+C_{4}\Delta t)^{k-1}L^{k}\Delta t^{k+k_{0}}.
\end{align*}
With $\Delta t< \frac{1}{\sqrt{3}}$, we find that
\begin{align*}
	2C_{4}C_{0}(C_{3}+C_{4}\Delta t)^{k-1}\leq 2C_{4}C_{0}(C_{3}+\frac{1}{\sqrt{3}}C_{4})^{k-1}:=\tilde{C_{1}}\\
	2C_{3}C_{0}(C_{3}+C_{3}\Delta t)^{k-1}\leq 2C_{3}C_{0}(C_{3}+\frac{1}{\sqrt{3}}C_{4})^{k-1}:=\tilde{C_{2}}
	\end{align*}
Therefore,
\begin{align*}
	&\|\Tt{X}-\Tt{X}^{k}\|\leq \tilde{C_{1}}L^{k} \Delta t^{k+k_0+1},\\
	&\|\Tt{V}-\Tt{V}^{k}\|\leq \tilde{C_{2}}L^{k} \Delta t^{k+k_0}.
\end{align*}	
which completes the proof.
\end{proof}

\bibliographystyle{siamplain}
\bibliography{references}
\end{document}